\documentclass[11pt]{amsart}
\usepackage{amssymb,amsfonts,amsmath,amsopn,amstext,amscd,latexsym,color}
\usepackage[all]{xy}

\input xy
\xyoption{all}

\newcommand{\bbN}{{\mathbb N}}

\newcommand{\bbQ}{{\mathbb Q}}

\newcommand{\bbZ}{{\mathbb Z}}

\newcommand{\cO}{{\mathcal O}}

\newcommand{\cU}{{\mathcal U}}

\newcommand{\Ll}{{\scriptscriptstyle{L}}}
\newcommand{\Kk}{{\scriptscriptstyle{K}}}

\newcommand{\Ff}{{\scriptscriptstyle{F}}}
\newcommand{\Hh}{{\scriptscriptstyle{H}}}
\newcommand{\LK}{{\scriptscriptstyle{L/K}}}
\newcommand{\FK}{{\scriptscriptstyle{F/K}}}
\newcommand{\MK}{{\scriptscriptstyle{M/K}}}
\newcommand{\EK}{{\scriptscriptstyle{E/K}}}
\newcommand{\LjK}{{\scriptscriptstyle{L_j/K}}}
\newcommand{\LLj}{{\scriptscriptstyle{L/L_j}}}
\newcommand{\LLs}{{\scriptscriptstyle{L/L_s}}}
\newcommand{\LHK}{{\scriptscriptstyle{L^{\Hh}/K}}}
\newcommand{\LsK}{{\scriptscriptstyle{L_s/K}}}
\newcommand{\Ls}{{\scriptscriptstyle{L_s}}}
\newcommand{\LH}{{\scriptscriptstyle{L^{\Hh}}}}
\newcommand{\LLH}{{\scriptscriptstyle{L/{L^{\Hh}}}}}
\newcommand{\KnrK}{{\scriptscriptstyle{K_{nr}/K}}}
\newcommand{\LKnr}{{\scriptscriptstyle{L/K_{nr}}}}

\usepackage[pdftex]{graphicx}

\author{Rachel Newton}
\title{Explicit local reciprocity for tame extensions}

\begin{document}

\begin{abstract}
We consider a tamely ramified abelian extension of local fields of degree $n$, without assuming the presence of the $n$th roots of unity in the base field. We give an explicit formula which computes the local reciprocity map in this situation.
\end{abstract}
\maketitle
\tableofcontents

\newpage

\section{Introduction}
Throughout this paper we call a field $K$ a `local field' if it is complete with respect to the topology induced by a discrete valuation and its residue field $k$ is finite. We write $\mathfrak{v}_{\Kk}$ for the discrete valuation on a field $K$, normalised so that the homomorphism $\mathfrak{v}_{\Kk}:K^*\rightarrow \bbZ$ is surjective. If $F/K$ is a finite extension with ramification degree $e=e_{\FK}$, we also use $\mathfrak{v}_{\Kk}$ to denote the unique extension to a valuation on $F$ with values in $\frac{1}{e}\bbZ$.

Let $L/K$ be a finite abelian extension of local fields. Let $N_{\LK}$ denote the norm map from $L$ to $K$. Then the local reciprocity map $\theta_{\LK}$, which will be defined in Section \ref{sectiondefnlocrec}, is a canonical isomorphism
$$\theta_{\LK}:\frac{K^*}{N_{\LK}(L^*)}\tilde{\longrightarrow} \mathrm{Gal}(L/K).$$
In the case where $L/K$ is an unramified extension, the local reciprocity map is completely determined by
$$\theta_{\LK}: \pi_{\Kk}\mapsto \mathrm{Frob}_{\LK}$$
where $\pi_{\Kk}$ is any uniformiser of $K$ and $\mathrm{Frob}_{\LK}\in\mathrm{Gal}(L/K)$ is the unique element which acts on the residue field $\ell$ of $L$ by $x\mapsto x^{|k|}$.

Now consider the case where $L/K$ is cyclic but the ramification index $e_{\LK}$ is non-trivial. To determine local reciprocity, we just need to find the image under $\theta_{\LK}$ of a generator of $K^*/N_{\LK}(L^*)$. However, in contrast with the unramified case, there is no canonical choice of generator.

The main result of this paper is the following formula which allows us to compute local reciprocity when the extension is abelian and tamely ramified.

\theoremstyle{plain} \newtheorem{generalthm}{Theorem}[section]
\begin{generalthm}
\label{generalthm}
Let $K$ be a local field with residue field $k$ of order $p^t$ and let $L/K$ be a tamely ramified abelian extension. Let $e=e_{\LK}$ denote the ramification index of $L/K$. Fix a uniformiser, $\pi_{\Kk}$, of $K$. Then the local reciprocity map $\theta_{\LK}$ satisfies
$$\frac{\theta_{\LK}\left(u\pi_{\Kk}^i\right)(\beta)}{\beta}\equiv \frac{\beta^{(p^{ti}-1)}}{((-1)^{(e-1)}\pi_{\Kk})^{(p^{ti}-1)\mathfrak{v}_{\Kk}(\beta)}u^{(p^t-1)\mathfrak{v}_{\Kk}(\beta)}}\ \ \mathrm{mod}\ \pi_{\Ll}$$
for all $i\in\bbN$, for all $u\in\cO_{\Kk}^*$ and for all $\beta\in L^*$.
\end{generalthm}

\theoremstyle{remark}\newtheorem{compute}[generalthm]{Remark}
\begin{compute}
\label{compute}
Since the extension $L/K$ is tame, these congruences for all $0\neq\beta\in\cO_{\Ll}$ are enough to uniquely determine an element of $\mathrm{Gal}(L/K)$ [for details, see Remark \ref{suffice}]. Thus if the extension $L/K$ is cyclic it suffices to apply the formula above to compute the image under $\theta_{\LK}$ of a generator of $K^*/{N_{\LK}(L^*)}$.
\end{compute}
The formula uses only the arithmetic of the extension $L/K$ and does not require the presence of any `extra' roots of unity in $K$. As we will see in Proposition \ref{ethroot}, tameness guarantees the presence in the ground field of roots of unity of degree equal to the ramification index - this will suffice for our purposes.

The existence of the local reciprocity map was originally deduced from the global theory in characteristic zero. This was the approach taken by Hasse-F.K.Schmidt in 1930 \cite{Hasse1}, \cite{Schmidt}. In his 1931 paper {\cite{Hasse2}, Hasse described the Brauer group of a local field and stated the Local Structure Theorem which says that every class in the Brauer group of a local field contains a cyclic (or Dickson-type) algebra. E. Noether hinted that these results could be used to derive local class field theory directly \cite{Roquette}.

In the 1950s, the proofs of local class field theory were rewritten in terms of Galois cohomology by Artin, Hochschild, Nakayama, Tate and Weil. Although elegant, the non-constructive formalism of this approach does not lend itself well to explicit computations. It is for this reason that during this paper we work with a definition of the local reciprocity map in terms of central simple algebras and the Brauer group.

There has been much previous work on explicit reciprocity laws. However, the approach taken by Artin-Hasse \cite{Artin-Hasse}, Iwasawa \cite{Iwasawa}, Coates-Wiles \cite{Coates-Wiles}, Wiles \cite{Wiles}, De Shalit \cite{De Shalit} and Fesenko-Vostokov \cite{Fesenko-Vostokov} is that of establishing an explicit formula for (analogues of) the Hilbert norm residue symbol. But the Hilbert symbol can only be used to calculate local reciprocity for an extension $L/K$ when $K$ contains a primitive root of unity of degree $[L:K]$.

Our general strategy for a cyclic extension $L/K$ is to choose generators of both $\mathrm{Gal}(L/K)$ and the quotient group $K^*/N_{\LK}(L^*)$ and construct the corresponding cyclic $K$-algebra. Then we explicitly construct an unramified extension $F/K$ which splits the cyclic algebra and use the Hasse invariant to compare our original choice of generator for $\mathrm{Gal}(L/K)$ with the canonical generator $\mathrm{Frob}_{\FK}\in\mathrm{Gal}(F/K)$.

Section \ref{mixedcase} covers the main content - the case where the extension $L/K$ is cyclic of prime-power degree and is neither totally ramified nor unramified. In this case we do not necessarily have a primitive root of unity of degree $[L:K]$ contained in $K$.

In Section \ref{totramcase}, we deal in a similar way with the case where the extension $L/K$ is cyclic of prime-power degree and totally ramified. In this case, $K$ contains a primitive root of unity of degree $[L:K]$ and an explicit computation of local reciprocity via the Hilbert norm residue symbol can be found, for instance in \cite{Serre}, Chapter XIV, \S3. Nevertheless, we include this case along with the unramified case in \mbox{Section \ref{unram case}} because the conclusion is that the same formula applies in all cases.

In Section \ref{abelian}, we decompose an abelian Galois group into a direct product of cyclic groups of prime-power order and combine the results of the previous sections to show that the congruences of Theorem \ref{generalthm} hold for an arbitrary tame abelian extension.

\theoremstyle{remark} \newtheorem{limitations}[generalthm]{Remark}
\begin{limitations}
\label{limitations}
In Proposition 5 of Chapter XV of \cite{Serre}, Serre describes an explicit formula which allows one to compute local reciprocity for a wildly ramified extension of prime degree.
As it stands, our method does not extend to the case where the extension $L/K$ is cyclic of prime-power degree and wildly ramified. In this case, Proposition \ref{ethroot} does not apply and we cannot assume the presence  in $K$ of roots of unity of degree equal to the ramification index. Furthermore, if the extension is wild, knowing $g(x)x^{-1}\ \mathrm{mod}\ \pi_{\Ll}$ for all $x\in L$ is not enough to specify an element $g\in\mathrm{Gal}(L/K)$ [c.f. Remark \ref{suffice}].
\end{limitations}

\paragraph{\textbf{Notation}}
For any local field $K$, we use the following notation:\\
\begin{tabular}{ll}
$\cO_{\Kk}$ &  $\textrm{the ring of integers of}\ K$\\
$\cO_{\Kk}^*$  & $\textrm{the units of}\ K$\\
$k$ & $\textrm{the residue field of}\ K$\\
$\mathfrak{v}_{\Kk}$  & $\textrm{the discrete valuation on}\ K\textrm{,}\ \textrm{normalised so}\ \mathfrak{v}_{\Kk}:K^*\twoheadrightarrow \bbZ$\\\\
\noalign{For a finite extension of local fields $L/K$, we use the following notation:}\\
$e=e_{\LK}$  & $\textrm{the ramification index of}\ L/K$\\
$f=f_{\LK}$ & $\textrm{the residue degree of}\ L/K$\\
$\mathfrak{v}_{\Kk}$ & $\textrm{the unique extension to a valuation on}\ L\ \textrm{with values in}\ \frac{1}{e}\bbZ$\\
\end{tabular}\\

We use the term `cyclic extension' to mean a (finite) Galois extension with cyclic Galois group.

For a finite unramified extension of local fields $L/K$, we denote by $\mathrm{Frob}_{\LK}$ the unique element of $\mathrm{Gal}(L/K)$ which acts on the residue field $\ell$ of $L$ by $x\mapsto x^{|k|}$.

\paragraph{\textbf{Acknowledgements}} I would like to thank my Ph.D. advisor Tim Dokchitser for his many helpful suggestions, encouragement and excellent guidance throughout this research. I am also very grateful to Vladimir Dokchitser for his generosity in giving his time for several enlightening discussions. I was supported by an EPSRC scholarship for the duration of this work.

\section{Background}
\subsection{A definition of the local reciprocity map}
\label{sectiondefnlocrec}
The purpose of this section is to define the local reciprocity map for finite abelian extensions via central simple algebras and the Brauer group.

Let $K$ be a local field. Fix a separable closure of $K$ and denote it by $K_s$. Write $G_{\Kk}=\mathrm{Gal}(K_s/K)$ for the absolute Galois group of $K$.\\
\paragraph{\textbf{Central simple algebras}}
\theoremstyle{definition}\newtheorem{csa}{Definition}[subsection]
\begin{csa}
\label{csa}
A finite dimensional algebra $A$ over $K$ is called:
\begin{itemize}
\item[(a)] \textbf{central} if its centre is $K$,
\item[(b)] \textbf{simple} if its only 2-sided ideals are $0$ and $A$.
\end{itemize}
\end{csa}
The Brauer group of $K$, denoted $\mathrm{Br}(K)$, classifies the central simple algebras over $K$.

\theoremstyle{plain}\newtheorem{A-W}[csa]{Theorem}
\begin{A-W}
\label{A-W}
Let $A$ be a central simple algebra over $K$. Then
$$A\cong M_n(D)$$
for some $n\in\bbN$ and some central division algebra $D$ over $K$. Moreover, $n$ and $D$ are determined by $A$.

\end{A-W}

\begin{proof}
This is a special case of the Artin-Wedderburn theorem which classifies semi-simple rings in terms of products of matrix rings over division rings.
\end{proof}

\theoremstyle{definition}\newtheorem{similar}[csa]{Definition}
\begin{similar}
\label{similar}
Let $A$ and $B$ be central simple algebras over $K$. Use the Artin-Wedderburn theorem to write $A\cong M_{n_1}(D_1)$, $B\cong M_{n_2}(D_2)$.

We call $A$ and $B$ \textbf{similar} if $D_1\cong D_2$.
\end{similar}

\theoremstyle{definition}\newtheorem{brauer}[csa]{Definition}
\begin{brauer}
\label{brauer}
The Brauer group of $K$, denoted $\mathrm{Br}(K)$, is the set of similarity classes of central simple algebras over $K$, with group structure given by the tensor product over $K$.
\end{brauer}
For the proof that this group structure is well-defined see, for instance, \cite{Milne}, Chapter IV, Section 2.

Now let $L/K$ be a finite extension. If $A$ is a central simple algebra over $K$ then $A\otimes_{\Kk} L$ is a central simple algebra over $L$.
\theoremstyle{definition}\newtheorem{splitting}[csa]{Definition}
\begin{splitting}
\label{splitting}
Let $A$ be a central simple algebra over $K$.
\begin{itemize}
\item[(i)] We say that a finite extension $L/K$ \textbf{splits} $A$ if $A\otimes_{\Kk}L\cong M_n(L)$ for some $n\in\bbN$.
\item[(ii)] We denote by $\mathrm{Br}(L/K)$ the subgroup of the Brauer group of $K$ consisting of all the similarity classes of central simple algebras over $K$ which are split by $L/K$. In other words, $\mathrm{Br}(L/K)$ is the kernel of the map $\mathrm{Br}(K) \rightarrow \mathrm{Br}(L)$ given by $A\mapsto A\otimes_{\Kk} L$.
\end{itemize}
\end{splitting}
It has been shown that any central simple algebra over a local field $K$ has a splitting field which is not only a finite extension of $K$ but is also \emph{unramified} over $K$. For Serre's proof see \cite{C-F}, Chapter VI, Appendix to Section 1. We restate this result in the following proposition:
\theoremstyle{plain}\newtheorem{unramsplit}[csa]{Proposition}
\begin{unramsplit}
\label{unramsplit}
$\mathrm{Br}(K)=\bigcup_{F} \mathrm{Br}(F/K)$ where the union is over all finite unramified extensions $F/K$ contained in $K_s$.
\end{unramsplit}

\paragraph{\textbf{Cyclic Algebras}}
\theoremstyle{definition}\newtheorem{cyclicalg}[csa]{Definition}
\begin{cyclicalg}
\label{cyclicalg}
Let $L/K$ be a cyclic extension of degree $n$. Fix an element $b\in K^*$ and a generator $\sigma$ of the Galois group $\mathrm{Gal}(L/K)$.
Define the corresponding \textbf{cyclic algebra} (which we denote temporarily as $A$) as follows:
$$A=\left\{\sum_{i=0}^{n-1}a_i v^i\ \Big|\ a_i\in L\right\}$$
with multiplication given by
\begin{itemize}
\item []$v^{n}=b$
\item []$v^i a=\sigma^i(a)v^i\ \ \forall a\in L$.
\end{itemize}
\end{cyclicalg}
One can check that $A$ is a central simple algebra of dimension $n^2$ over $K$. Furthermore, since $L$ is contained in $A$ as a maximal commutative subalgebra, $A$ is split by $L$.

Let
\begin{eqnarray*}
\chi:G_{\Kk}\twoheadrightarrow\mathrm{Gal}(L/K) & \cong & \frac{\frac{1}{n}\bbZ}{\bbZ}\subset \frac{\bbQ}{\bbZ}\\
                                                                \sigma & \mapsto & \frac{1}{n}\ \mathrm{mod}\ \bbZ
\end{eqnarray*}
be the continuous character of $G_{\Kk}=\mathrm{Gal}(K_s/K)$ which factors through $\mathrm{Gal}(L/K)$ and sends $\sigma$ to $\frac{1}{n}\ \mathrm{mod}\ \bbZ$.

Note that specifying $\chi$ is equivalent to specifying both $\sigma$ and the extension $L/K$. So it makes sense to use the notation $$A=(\chi, b)$$
for the cyclic algebra defined above.
In an abuse of notation, we will in fact use $(\chi, b)$ to denote both $A$ and its class $[A]$ in the Brauer group $\mathrm{Br}(K)$.

\theoremstyle{plain}\newtheorem{surj}[csa]{Proposition}
\begin{surj}
\label{surj}
Let $L$, $\sigma$ and $\chi$ be as above. Then the map
\begin{eqnarray*}
K^* & \rightarrow & \mathrm{Br}(L/K)\\
b & \mapsto & (\chi,b)
\end{eqnarray*}
is surjective.

Moreover, $(\chi, b_1)=(\chi, b_2)$ if and only if $b_1b_2^{-1}\in N_{\LK}(L^*)$.
\end{surj}
\begin{proof}
See \cite{Fisher} where this is Proposition 5.5.
\end{proof}

The cyclic algebra construction gives a map
$$(\ ,\ ):\mathrm{Hom}_{\scriptstyle{cts}}\left(G_{\Kk},{\bbQ}/{\bbZ}\right)\times K^*\longrightarrow \mathrm{Br}(K).$$

\newtheorem{cupprod}[csa]{Proposition}
\begin{cupprod}
\label{cupprod}
The map $(\ ,\ )$ coming from the cyclic algebra construction can be identified with the cup product
$$\mathrm{H}^2(G_{\Kk},\bbZ)\times \mathrm{H}^0(G_{\Kk}, {K_s}^*)\longrightarrow \mathrm{H}^2(G_{\Kk}, {K_s}^*)$$
induced by $\bbZ\times {K_s}^*\rightarrow {K_s}^*;\ (m,z)\mapsto z^m$.
In particular, it is \textbf{bilinear}.
\end{cupprod}
\begin{proof}
Since $\bbQ$ is cohomologically trivial, we can identify $\mathrm{H}^2(G_{\Kk},\bbZ)$ with $\mathrm{Hom}_{\scriptstyle{cts}}\left(G_{\Kk},{\bbQ}/{\bbZ}\right)=\mathrm{H}^1(G_{\Kk},{\bbQ}/{\bbZ})$ via the boundary map coming from the exact sequence
\begin{displaymath}
\xymatrix{
0\ar[r] & \bbZ\ar[r] & \bbQ\ar[r] & {\bbQ}/{\bbZ}\ar[r] & 0
}
\end{displaymath}
Using the explicit isomorphism between $\mathrm{Br}(K)$ and $\mathrm{H}^2(G_{\Kk}, {K_s}^*)$ given in \cite{Milne} Chapter IV \S3 and the cup product computation in Lemma 1 of the Appendix on p176 of \cite{Serre}, one can check that the relevant diagram commutes.
\end{proof}

\newtheorem{kercupprod}[csa]{Corollary}
\begin{kercupprod}
For any character $\chi\in \mathrm{Hom}_{\scriptstyle{cts}}\left(G_{\Kk},{\bbQ}/{\bbZ}\right)$ we have
$$(\chi,b)=0 \iff b\in N_{\LK}(L^*)$$
where $L$ is the fixed field of the kernel of $\chi$.
\end{kercupprod}
\begin{proof}
This is an immediate consequence of Proposition \ref{surj} and bilinearity.
\end{proof}

\paragraph{\textbf{The Hasse invariant}}
\newtheorem{inv1}[csa]{Proposition}
\begin{inv1}
\label{inv1}
Let $F/K$ be an unramified extension of degree $n$. Then the map $\mathrm{inv}_{\FK}$ defined as
\begin{eqnarray*}
\mathrm{inv}_{\FK} :  \mathrm{Br}(F/K) & \longrightarrow & \frac{\frac{1}{n}\bbZ}{\bbZ}\ \ \subset\ \frac{\bbQ}{\bbZ}\\
   (\chi, b) & \mapsto & \mathfrak{v}_{\Kk}(b)\chi(\mathrm{Frob}_{\FK})
\end{eqnarray*}
is a well-defined isomorphism.
\end{inv1}
\begin{proof}
See \cite{Milne}, Chapter IV, \S4.
\end{proof}
Now suppose we have a tower of unramified extensions $M/L/K$. Clearly there is a natural inclusion $\mathrm{Br}(L/K)\hookrightarrow\mathrm{Br}(M/K)$. Moreover, one can check that the diagram below commutes.
\begin{displaymath}
\xymatrix{
\mathrm{Br}(M/K)\ar[rr]^{\mathrm{inv}_{\MK}} && {\bbQ}/{\bbZ}\\
\mathrm{Br}(L/K)\ar[rr]^{\mathrm{inv}_{\LK}}\ar@{^{(}->}[u] && {\bbQ}/{\bbZ}\ar@{=}[u]
}
\end{displaymath}
Recall that $\mathrm{Br}(K)=\bigcup_{F} \mathrm{Br}(F/K)$ where the union is over all finite unramified extensions $F/K$ contained in $K_s$. Thus we can combine the maps for each unramified extension to get a (canonical) isomorphism
\begin{displaymath}
\xymatrix{
\mathrm{inv}_{\Kk}: \mathrm{Br}(K)\ar[r]^{\ \ \ \sim} & {\bbQ}/{\bbZ}.}
\end{displaymath}
\theoremstyle{definition}\newtheorem{inv2}[csa]{Definition}
\begin{inv2}
\label{inv2}
We call the map $\mathrm{inv}_{\Kk}$ the \textbf{Hasse invariant}.
\end{inv2}

\paragraph{\textbf{The local reciprocity map}}
\theoremstyle{plain}\newtheorem{locrec}[csa]{Proposition-Definition}
\begin{locrec}
\label{locrec}
Let $L/K$ be a finite abelian extension. There is a canonical isomorphism
$$\theta_{\LK}:\frac{K^*}{N_{\LK}(L^*)}\tilde{\longrightarrow} \mathrm{Gal}(L/K)$$
uniquely determined by
$$\mathrm{inv}_{\Kk}(\chi, b)=\chi\left(\theta_{\LK}\left(b\right)\right) \ \ \ \ \ \forall\ \chi\in \mathrm{Hom}\left(\mathrm{Gal}(L/K), {\bbQ}/{\bbZ}\right),\ \forall\ b\in K^*.$$
We call $\theta_{\LK}$ the \textbf{local reciprocity map}.
\end{locrec}

\begin{proof}
See \cite{Serre}, Chapter XI, Proposition 2.
\end{proof}

\newtheorem{unramlocrec}[csa]{Corollary}
\begin{unramlocrec}
\label{unramlocrec}
Let $L/K$ be an unramified extension. Then the local reciprocity map satisfies
$\theta_{\LK}(b)=\mathrm{Frob}_{\LK}^{\mathfrak{v}_{\Kk}(b)}$
for all $b\in K^*$.
\end{unramlocrec}

\begin{proof}
This is clear from the definition of the Hasse invariant.
\end{proof}

\newtheorem{diagram}[csa]{Proposition}
\begin{diagram}
\label{diagram}
Let $L/K$ be a finite abelian extension and let $E/K$ be any finite extension. Then the following diagram commutes.

\begin{displaymath}
\xymatrix{
E^*\ar[rr]^{\theta_{LE/E}}\ar[d]_{N_{\EK}} && \mathrm{Gal}(LE/E)\ar@{^{(}->}[d]\\
K^*\ar[rr]^{\theta_{L/K}} && \mathrm{Gal}(L/K)
}
\end{displaymath}

\end{diagram}
\begin{proof}
See the second diagram in Section 2.4 of Chapter VI of \cite{C-F}.
\end{proof}

\subsection{Some properties of tame abelian extensions of local fields}
Let $K$ be a local field and let $L/K$ be a finite Galois extension. Write $K_{nr}$ for the maximal unramified extension of $K$ inside $L$. Let \mbox{$G=\mathrm{Gal}(L/K)$}.
\theoremstyle{definition}\newtheorem{ramgp}{Definition}[subsection]
\begin{ramgp}
\label{ramgp}
Let $i$ be an integer with $i\geq -1$. Define the \textbf{ith ramification group}, $G_i$, as follows:$$G_i=\{g\in G\mid \mathfrak{v}_{\scriptscriptstyle{L}}(g(z)-z)\geq i+1\quad \forall z\in \mathcal{O}_L\}$$
We call $G_0$ the \textbf{inertia subgroup} of $G$.
\end{ramgp}
The $G_i$ form a decreasing sequence of normal subgroups of $G$ such that $G_{-1}=G$ and $G_i=\{1\}$ for $i$ sufficiently large. See \cite{Serre} Chapter IV, Proposition 1.
\theoremstyle{plain}\newtheorem{G0}[ramgp]{Proposition}
\begin{G0}
\label{G0}
$G_0=\mathrm{Gal}(L/K_{nr})$ and therefore $e=e_{\LK}=|G_0|$.
\end{G0}

\begin{proof}
See \cite{Serre}, Chapter IV, Corollary to Proposition 2.
\end{proof}

\newtheorem{tame}[ramgp]{Definition}
\begin{tame}
\label{tame}
A Galois extension $L/K$ is called \textbf{tamely ramified} (or simply \textbf{tame}) if $G_1=\{1\}$. Otherwise the extension is called \textbf{wildly ramified}.
\end{tame}
Now let $F$ be any local field. Denote by $\mathfrak{f}$ the residue field of $F$. Let $\pi_{\Ff}$ be any uniformiser of $F$. Following Serre in \cite{Serre}, we define a \emph{filtration} of $\cO_{\Ff}^*$, the group of units of $F$, by:
\begin{eqnarray*}
U_{\Ff}^0 & = & \cO_{\Ff}^*\\
U_{\Ff}^i & = & 1+\pi_{\Ff}^i\cO_{\Ff}\ \ \mathrm{for}\ i\geq 1.
\end{eqnarray*}
These subgroups form a neighbourhood base of $1$ for the topology induced on $\cO_{\Ff}^*$ by the discrete valuation on $F^*$.

\theoremstyle{plain}\newtheorem{Ui's}[ramgp]{Proposition}
\begin{Ui's}
\label{Ui's}

\begin{itemize}
\item[(a)] ${U_{\Ff}^0}/{U_{\Ff}^1}=\mathfrak{f}^*$.
\item[(b)] For $i\geq 1$, the group $U_{\Ff}^i/U_{\Ff}^{i+1}$ is (non-canonically) isomorphic to the additive group of $\mathfrak{f}$.
\end{itemize}
\end{Ui's}
\begin{proof}
See \cite{Serre}, Chapter IV, where this is Proposition 6.
\end{proof}

\newtheorem{quotients}[ramgp]{Proposition}
\begin{quotients}
\label{quotients}
Pick a uniformiser $\pi_{\Ll}$ of $L$. For $i\geq 0$, the map
\begin{eqnarray*}
G_i & \longrightarrow & U_{\Ll}^i\\
g & \mapsto & g(\pi_{\Ll})/\pi_{\Ll}
\end{eqnarray*}
defines by passage to the quotient an injective homomorphism $$G_i/G_{i+1}\hookrightarrow U_{\Ll}^i/U_{\Ll}^{i+1}.$$
This homomorphism is independent of the choice of $\pi_{\Ll}$.
\end{quotients}

\begin{proof}
This is Proposition 7 of Chapter IV of \cite{Serre}.
\end{proof}

\newtheorem{G1}[ramgp]{Corollary}
\begin{G1}
\label{G1}
Let $p$ be the characteristic of the residue field $k$. Then $G_1$ is the unique Sylow $p$-subgroup of $G_0$. In particular, the extension $L/K$ is tamely ramified if and only if $e=e_{\LK}=|G_0|$ is coprime to $p$.
\end{G1}

\begin{proof}
This follows directly from Propositions \ref{Ui's} and \ref{quotients} and the fact that $G_i=\{1\}$ for $i$ sufficiently large. Consider successive quotients and recall that the residue field $\ell$ is a $p$-group and $\ell^*$ has order coprime to $p$.
\end{proof}

The following result must exist in the literature but we are currently unable to find it stated in this form.
\newtheorem{ethroot}[ramgp]{Proposition}
\begin{ethroot}
\label{ethroot}
Let $L/K$ be a tamely ramified abelian extension of local fields. Then $e=e_{\LK}$ divides $|k^*|$ and $K$ contains all $e$th roots of unity.
\end{ethroot}

\begin{proof}
In the diagram below, the rows are exact and all squares commute.
\begin{displaymath}
\xymatrix{
0\ar[r] & \frac{N_{\LK}(L^*)\cap \cO_{\Kk}^*}{1+\pi_{\Kk}\cO_{\Kk}}\ar[r]\ar@{^{(}->}[d] & \frac{N_{\LK}(L^*)}{1+\pi_{\Kk}\cO_{\Kk}}\ar[r]^{\ \ \ \ \mathfrak{v}_{\Kk}}\ar@{^{(}->}[d] & f\bbZ\ar[r]\ar@{^{(}->}[d] & 0\\
0\ar[r] & \frac{\cO_{\Kk}^*}{1+\pi_{\Kk}\cO_{\Kk}}\ar[r] & \frac{K^*}{1+\pi_{\Kk}\cO_{\Kk}}\ar[r]^{\ \ \ \ \mathfrak{v}_{\Kk}} & \bbZ\ar[r] & 0}
\end{displaymath}
Applying the Snake Lemma, we get the following exact sequence:
\begin{displaymath}
\xymatrix{
0\ar[r] & \frac{\cO_{\Kk}^*}{N_{\LK}(L^*)\cap \cO_{\Kk}^*}\ar[r] & \frac{K^*}{N_{\LK}(L^*)}\ar[r] & \frac{\bbZ}{f\bbZ}\ar[r] & 0
}
\end{displaymath}
The existence of the local reciprocity map tells us that
$$\left|\frac{K^*}{N_{\LK}(L^*)} \right|=\left|\mathrm{Gal}(L/K)\right|=ef$$
so the exact sequence allows us to deduce that
$$\left[\frac{\cO_{\Kk}^*}{1+\pi_{\Kk}\cO_{\Kk}} : \frac{N_{\LK}(L^*)\cap \cO_{\Kk}^*}{1+\pi_{\Kk}\cO_{\Kk}} \right]=\left|\frac{\cO_{\Kk}^*}{N_{\LK}(L^*)\cap \cO_{\Kk}^*}\right|=e.$$
Recalling from Proposition \ref{Ui's} that $\frac{\cO_{\Kk}^*}{1+\pi_{\Kk}\cO_{\Kk}}$ is isomorphic to $k^*$, we conclude that $e$ divides $|k^*|$. Therefore the cyclic group $k^*$ contains a subgroup of size $e$, i.e. a copy of the $e$th roots of unity. $L/K$ is tamely ramified so $p$ is coprime to $e$ by Corollary \ref{G1}. So we can use Hensel's Lemma to lift the $e$th roots of unity in $k^*$ to the $e$th roots of unity in $K$.
\end{proof}

\newtheorem{U1norms}[ramgp]{Proposition}
\begin{U1norms}
\label{U1norms}
Let $L/K$ be a tamely ramified abelian extension of local fields with $[L:K]$ coprime to $p=\mathrm{char}(k)$. Then we have
$$U_{\Kk}^1=1+\pi_{\Kk}\cO_{\Kk}\subset N_{\LK}(L^*).$$
\end{U1norms}

\begin{proof}
We will show that $U_{\Kk}^1\cap N_{\LK}(L^*)=U_{\Kk}^1$. The local existence theorem of class field theory tells us that $N_{\LK}(L^*)$ is an open subgroup of $K^*$.\footnote{See, for example, \cite{Serre}, Chapter XIV, Theorem 1 and its first Corollary.} Hence we must have
\begin{eqnarray*}
U_{\Kk}^m\subset U_{\Kk}^1\cap N_{\LK}(L^*)\subset U_{\Kk}^1
\end{eqnarray*}
for $m$ sufficiently large.
Hence $$[U_{\Kk}^1:U_{\Kk}^1\cap N_{\LK}(L^*)][U_{\Kk}^1\cap N_{\LK}(L^*):U_{\Kk}^m]=[U_{\Kk}^1:U_{\Kk}^m].$$ But we know from Proposition \ref{Ui's} that $[U_{\Kk}^1:U_{\Kk}^m]$ is $|k|^{m-1}$, in particular a power of $p$. Therefore $[U_{\Kk}^1:U_{\Kk}^1\cap N_{\LK}(L^*)]$ is also a power of $p$.

On the other hand, we have
$$\frac{U_{\Kk}^1}{U_{\Kk}^1\cap N_{\LK}(L^*)} \hookrightarrow \frac{K^*}{N_{\LK}(L^*)}.$$
The existence of the local reciprocity map tells us that the right-hand side is isomorphic to $\mathrm{Gal}(L/K)$, which has size $[L:K]$. Thus $\left|\frac{U_{\Kk}^1}{U_{\Kk}^1\cap N_{\LK}(L^*)}\right|$ divides $[L:K]$. But $[L:K]$ is coprime to $p$ and therefore $\left|\frac{U_{\Kk}^1}{U_{\Kk}^1\cap N_{\LK}(L^*)}\right|=1$.
\end{proof}

\section{The cyclic mixed case}
\label{mixedcase}
Let $K$ be a local field and let $L/K$ be a tamely ramified cyclic extension of prime-power degree. Write $K_{nr}$ for the maximal unramified extension of $K$ inside $L$. In this section, we derive a formula which allows us to explicitly compute local reciprocity in the case where the extension $L/K$ is neither totally ramified nor unramified. We do not assume that $K$ contain a primitive root of unity of degree $[L:K]$. For this reason, our work does not follow from the explicit formulae for the Hilbert norm residue symbol given by Iwasawa \cite{Iwasawa}, Serre \cite{Serre}, Fesenko-Vostokov \cite{Fesenko-Vostokov}, De Shalit \cite{De Shalit} and others. We will see in Sections \ref{totramcase}, \ref{unram case} and \ref{abelian} that in fact our formula applies more generally. We summarise our current situation as follows:
\subsection{Situation}
\label{situation}
\begin{itemize}
\item[-] the residue field of $K$, denoted by $k$, has $|k|=p^t$
\item[-]$L/K$ cyclic extension of local fields with $[L:K]=q^{n}$ where $q$ and $p$ are distinct primes
\item[-] the ramification index of the extension $L/K$ satisfies $e=e_{\LK}=q^{d}$ with $1<q^{d}<q^n$
\end{itemize}
Pick a uniformiser, $\pi_{\Ll}$, of $L$. Now $K(\pi_{\Ll})$ is a subfield of $L$ and therefore, since $\mathrm{Gal}(L/K)$ is cyclic of prime-power order, $K(\pi_{\Ll})$ can be placed in the vertical tower below.
\begin{displaymath}
\xymatrix{
L \ar@{-}[d]^{q^d} \\
K_{nr}\ar@{-}[d]^{q^{n-d}}\\
K}
\end{displaymath}
In other words, either $K(\pi_{\Ll})\subset K_{nr}$ or $K_{nr}\subset K(\pi_{\Ll})$. But the extension $K(\pi_{\Ll})/K$ has ramification index $q^d>1$ so cannot be contained in an unramified extension of $K$. Hence we must have $K_{nr}\subset K(\pi_{\Ll})$. But now the ramification index and the residue degree of $K(\pi_{\Ll})/K$ are equal to those of $L/K$. Therefore we deduce that $L=K(\pi_{\Ll})$.

We have $\pi_{\Ll}^{q^d}=u\pi_{\Kk}$ for some uniformiser $\pi_{\Kk}$ of $K$ and some $u\in\cO_{\Ll}^*$. Reducing $\mathrm{mod}\ \pi_{\Ll}$ and recalling that the residue field of $L$ is the same as that of $K_{nr}$, we see that $u=u_1u_0$ where $u_1\equiv 1\ \mathrm{mod}\ \pi_{\Ll}$ and $u_0\in \cO_{\Kk_{nr}}^*$.

Now consider the polynomial $x^{q^d}-u_1$, which reduces to $x^{q^d}-1\ \mathrm{mod}\ \pi_{\Ll}$.
We know from Proposition \ref{ethroot} that the $q^d$th roots of unity are contained in $K$ and therefore in $L$. Since $q^d$ is coprime to the residue characteristic, we can apply Hensel's lemma and lift the roots of the polynomial $\mathrm{mod}\ \pi_\Ll$ to $q^d$th roots of $u_1$ in $L$.

Extracting one of these roots from the equality $\pi_{\Ll}=\sqrt[q^d]{u\pi_{\Kk}}=\sqrt[q^d]{u_1u_0\pi_{\Kk}}$, we see that $\alpha=\sqrt[q^d]{u_0\pi_{\Kk}}\in L$ is a uniformiser for $L$. Therefore, by the argument above, we can write $L=K(\alpha)$. Our situation is the following:
\begin{displaymath}
\xymatrix{
L=K(\alpha) \ar@{-}[d]_{q^d} \\
K_{nr}\ar@{-}[d]\\
K(u_0)\ar@{-}[d]\\
K}
\end{displaymath}

\subsection{A cyclic algebra}

Choose a generator, $\sigma$, of $G=\mathrm{Gal}(L/K)$ and let $\chi\in \mathrm{Hom}(G,\bbQ/{\bbZ})$ be the unique character such that $\chi : \sigma\mapsto\frac{1}{q^n}\ \mathrm{mod}\ \bbZ$. Construct the cyclic algebra $A=(\chi,\pi_{\Kk})$ as follows:

$$ A=(\chi, \pi_{\Kk})=\left\{\sum_{i=0}^{q^n-1}a_i v^i\ \Big|\ a_i\in L=K(\alpha)\right\}$$
with multiplication given by

\begin{itemize}
\item []$v^{q^n}=\pi_{\Kk}$
\item []$v^i a=\sigma^i(a)v^i\ \ \forall a\in L$.
\end{itemize}
$A$ is a central simple algebra over $K$ which is split by $L/K$.  In order to determine the local reciprocity map $\theta_{\LK}$, it will suffice to compute the Hasse invariant of $[A]$. This follows from the definition of $\theta_{\LK}$ given at the end of Section \ref{sectiondefnlocrec} and the fact that the class of $\pi_{\Kk}$ generates ${K^*}/{N_{\LK}(L^*)}$. [The latter is verified in Section \ref{mixedlocrec} where we show that $\theta_{\LK}({\pi_{\Kk}})$ generates $\mathrm{Gal}(L/K)$.]

In order to compute the Hasse invariant of $[A]$, we need to express it as a cyclic algebra coming from an \emph{unramified} extension of $K$. This is possible by Propositions \ref{unramsplit} and \ref{cupprod}. The first step is to find an unramified splitting field for $A$.

\subsection{An unramified splitting field for $A$}Define $\gamma=\alpha v^{-q^{n-d}}$, where $q^{n-d}=f_{\LK}$ is the residue degree of the extension $L/K$. Consider the field $F=K(\gamma)$.

\newtheorem{unram}{Lemma}[subsection]
\begin{unram}
\label{unram}
$F/K$ is an unramified Galois extension.
\end{unram}

\begin{proof}

First we claim that $\gamma^{q^d}=(-1)^{(q^d-1)}\frac{\alpha^{q^d}}{\pi_{\Kk}}=(-1)^{(q^d-1)}u_0\in \cO_{\Kk_{nr}}^*$. In fact, using the multiplication rules detailed above, we see that:
\begin{eqnarray*}
\gamma^{q^d}&=&(\alpha v^{-{q^{n-d}}})^{q^d}\\
 &=&\alpha v^{-{q^{n-d}}}\alpha v^{-{q^{n-d}}}\dots\alpha v^{-{q^{n-d}}}\\
 &=&\alpha\sigma^{-{q^{n-d}}}(\alpha)\sigma^{-2{q^{n-d}}}(\alpha)\dots \sigma^{-(q^d-1)q^{n-d}}(\alpha)(v^{-q^{n-d}})^{q^d}.
\end{eqnarray*}
Now observe that $\sigma^{-{q^{n-d}}}$ generates the unique subgroup of order $q^d$ inside $\mathrm{Gal}(L/K)$ and that this is the subgroup which fixes the intermediate field $K_{nr}$. So the $\sigma^{-i{q^{n-d}}}(\alpha)$ for $0\leq i\leq q^d-1$ are precisely the roots of the polynomial $x^{q^{d}}-u_0\pi_{\Kk}$, which is the minimal polynomial of $\alpha$ over $K_{nr}$. But then the product of the roots must be $(-1)^{(q^d-1)}u_0\pi_{\Kk}$ and, on observing that $(v^{-q^d})^{q^{n-d}}=v^{-q^n}=\pi_{\Kk}^{-1}$, we conclude that $\gamma^{q^d}=(-1)^{(q^d-1)}u_0$ as required.

But now, since $u_0\in\cO_{\Kk_{nr}}^*$ and $q^d$ is coprime to the residue characteristic, Hensel's lemma ensures that $F/K$ is an unramified extension. The residue field $k$ of $K$ is finite so the residue field $\mathfrak{f}$ of $F$ is automatically a Galois extension of $k$. But then $F/K$ must be Galois since it is an unramified extension and we have $\mathrm{Gal}(F/K)\cong\mathrm{Gal}(\mathfrak{f}/k)$.
\end{proof}
Note that for any $i\in \bbN$
$$(v^i\gamma v^{-i})^{q^d}=v^i\gamma^{q^d} v^{-i}=\sigma^i(u_0).$$
So $v^i\gamma v^{-i}$ is a $q^d$th root of a Galois conjugate of $u_0$ and thus a Galois conjugate of $\gamma$.

Let $\tau\in\mathrm{Gal}(F/K)$ denote the generator which maps any $x\in F$ to $v x v^{-1}$. Let $\phi\in \mathrm{Hom}(\mathrm{Gal}(F/K),\bbQ/{\bbZ})$ be the unique character such that \mbox{$\phi:\tau\mapsto\frac{1}{q^n}\ \mathrm{mod}\ \bbZ$}. Recall the definition of the cyclic algebra $A=(\chi,\pi_{\Kk})$ and rewrite $A$ in the following way:
\begin{eqnarray*}
A  =  (\chi, \pi_{\Kk}) & = & \left\{\sum_{i=0}^{q^n-1}a_i v^i\ \Big|\ a_i\in L=K(\alpha)\right\} \\
   & = & \left\{\sum_{i=0}^{q^n-1}c_i v^i\ \Big|\ c_i\in F=K(\gamma)\right\}= (\phi,\pi_{\Kk})
\end{eqnarray*}
where
\begin{itemize}
\item []$v^{q^n}=\pi_{\Kk}$
\item []$v^i c=\tau^i(c)v^i\ \ \forall c\in F$.
\end{itemize}
Note that considering the dimension of $A$ as a vector space over $K$ gives $$q^{2n}=q^n[L:K]=[A:K]=[A:F][F:K]=q^n[F:K]$$
which justifies the implicit assumption that $[F:K]=q^n$.

We can now use this to compute the Hasse invariant of $A$:
$$\mathrm{inv}_{\Kk}(\chi,\pi_{\Kk})=\mathrm{inv}_{\Kk}(\phi,\pi_{\Kk})=\mathfrak{v}_{\Kk}(\pi_{\Kk})\phi(\mathrm{Frob}_{\FK})$$
where $\mathfrak{v}_{\Kk}$ denotes the valuation on $K$, normalised so that $\mathfrak{v}_{\Kk}(\pi_{\Kk})=1$. Observe that $\phi(\mathrm{Frob}_{\FK})=\frac{r}{q^n}$ where $r\in \bbN$ is such that $\tau^r=\mathrm{Frob}_{\FK}$.

\subsection{Determining local reciprocity}
\label{mixedlocrec}
\newtheorem{tau-sigma}{Lemma}[subsection]
\begin{tau-sigma}
\label{tau-sigma}
For all $i\in\bbN$, $\tau^i(\gamma)\gamma^{-1}=\sigma^i(\alpha)\alpha^{-1}$.
\end{tau-sigma}

\begin{proof}
Recall that for all $i\in\bbN\ $, $\tau^i(\gamma)=v^i\gamma v^{-i}$ (by definition of $\tau$). Thus
\begin{eqnarray*}
\tau^i(\gamma)\gamma^{-1} & = & v^i\gamma v^{-i}\gamma^{-1}\\
 & = & v^i (\alpha v^{-q^{n-d}})  v^{-i} v^{q^{n-d}} \alpha^{-1}\\
 & = & v^i\alpha v^{-i} \alpha^{-1}\\
 & = & \sigma^i(\alpha)\alpha^{-1}
\end{eqnarray*}
\end{proof}
Recall that in order to determine the Hasse invariant for $A$, we are looking for an $r\in \bbN$ such that $\tau^r=\mathrm{Frob}_{\FK}$.
In other words, $\tau^r$ should have the same action as $\mathrm{Frob}_{\FK}$ on the residue field of $F$. We need $r\in\bbN$ such that:
$$\tau^r(\gamma)\gamma^{-1}\equiv\gamma^{p^t-1}\equiv ((-1)^{(q^d-1)}u_{0})^{\frac{p^t-1}{q^d}}\ \ \mathrm{mod}\ \pi_{\Kk}$$
where $p^t=|k|$ is the number of elements in the residue field of $K$. We saw in Proposition \ref{ethroot} that $e=e_{\LK}=q^d$ divides $p^t-1$ since the group of $e$th roots of unity injects into $k^*$.

Now apply Lemma \ref{tau-sigma} to get the equivalent condition
\begin{equation}
\label{eq:1st congruence}
\sigma^r(\alpha)\alpha^{-1}\equiv ((-1)^{(e-1)}u_{0})^{\frac{p^t-1}{e}}\equiv \alpha^{p^t-1}((-1)^{(e-1)}\pi_{\Kk})^{\frac{1-p^t}{e}} \ \ \mathrm{mod}\ \pi_{\Kk}.
\end{equation}
The local reciprocity map
$$\theta_{L/K}: \frac{{K^*}}{{N_{L/K}(L^*)}}\rightarrow \mathrm{Gal}(L/K)$$
is the isomorphism which is uniquely determined by $$\mathrm{inv}_{\Kk}(\psi,b)=\psi(\theta_{L/K}(b))$$ for all $b\in K^*$ and all $\psi\in \mathrm{Hom}(\mathrm{Gal}(L/K),\bbQ/\bbZ)$.

Thus, taking $b=\pi_{\Kk}$ and  $\psi=\chi$, where $\chi:\sigma\mapsto \frac{1}{q^n}$, we see that $\theta_{L/K}$ is determined by the following action on ${\pi_{\Kk}}$:
$$\theta_{L/K}:\pi_{\Kk}\mapsto \sigma^r$$
where $r$ is as described above.

Note that $\tau^r=\mathrm{Frob}_{\FK}$ is another generator of $\mathrm{Gal}(F/K)$. Hence $r$ must be coprime to $[F:K]=q^n=[L:K]$. Thus $\sigma^r$ is a generator of $\mathrm{Gal}(L/K)$.

\subsection{Congruence conditions}
\label{congruence conditions}
From now on we will write $\sigma_{b}$ for the image $\theta_{\LK}(b)$ of an element $b\in K^*$ under the local reciprocity map $\theta_{\LK}$. Hence we write the congruence \eqref{eq:1st congruence} described above as follows:
$$\sigma_{\pi_{\Kk}}(\alpha)\alpha^{-1}\equiv \alpha^{p^t-1}((-1)^{(e-1)}\pi_{\Kk})^{\frac{1-p^t}{e}}\ \ \mathrm{mod}\ \pi_{\Kk}$$
where $\alpha$ is the chosen uniformiser of $L$ which generates $L$ over $K$, as described in Section \ref{situation}.

\newtheorem{general element}{Lemma}[subsection]
\begin{general element}
\label{general element}
For all $\beta\in L^*$
\begin{equation} \label{eq:general element}
\frac{\sigma_{\pi_{\Kk}}(\beta)}{\beta}\equiv \frac{\beta^{(p^t-1)}}{((-1)^{(e-1)}\pi_{\Kk})^{(p^t-1)\mathfrak{v}_{\Kk}(\beta)}}\ \ \mathrm{mod}\ \pi_{\Ll}
\end{equation}
where $\pi_{\Ll}$ denotes a uniformiser of $L$ and, in an abuse of notation, we write $\mathfrak{v}_{\Kk}$ for the extension to $L$ of the valuation on $K$, normalised so that \mbox{$\mathfrak{v}_{\Kk}(\pi_{\Kk})=1$}.
\end{general element}

\begin{proof}
First we show that the claim holds for an arbitrary uniformiser $\pi_{\Ll}$ of $L$. By the description in Section \ref{situation} of how to produce $\alpha$ from an arbitrary choice of uniformiser, we can assume that $\pi_{\Ll}=w\alpha$ for some $w\in\cO_{\Ll}^*$ satisfying $w^e \equiv 1\ \mathrm{mod}\ \pi_{\Ll}$. Hence $w$ is congruent to some $e$th root of unity modulo $\pi_{\Ll}$. But these are all in $K$ so are fixed by any element of $\mathrm{Gal}(L/K)$. Therefore $\sigma_{\pi_{\Kk}}(w)\equiv w\ \mathrm{mod}\ \pi_{\Ll}$ and so we have
\begin{eqnarray*}
 \frac{\sigma_{\pi_{\Kk}}(\pi_{\Ll})}{{\pi_{\Ll}}} & = & \frac{\sigma_{\pi_{\Kk}}(w)}{w}\frac{\sigma_{\pi_{\Kk}}(\alpha)}{\alpha}\\
& \equiv & \frac{\sigma_{\pi_{\Kk}}(\alpha)}{\alpha}\\
 & \equiv & \frac{\alpha^{p^t-1}}{((-1)^{(e-1)}\pi_{\Kk})^{(p^t-1)\mathfrak{v}_{\Kk}(\alpha)}}\ \ \mathrm{mod}\ \pi_{\Ll}.
\end{eqnarray*}
Observe that, since $e$ divides $p^t-1$, $\pi_{\Ll}^{p^t-1}=w^{p^t-1}\alpha^{p^t-1}\equiv \alpha^{p^t-1}\ \mathrm{mod}\ \pi_{\Ll}$. Putting this together with the fact that $\mathfrak{v}_{\Kk}(\pi_{\Ll})=\mathfrak{v}_{\Kk}(\alpha)$ we conclude that the congruence \eqref{eq:general element} holds for $\beta=\pi_{\Ll}$.

Note that any unit in $\cO_{\Ll}^*$ can be expressed as a quotient of two uniformisers of $L$. Now observe that if the congruence \eqref{eq:general element} holds for two elements $\beta_1,\beta_2\in L^*$ then it also holds for their product or quotient. These last two remarks show that the general case follows immediately from the above computation for a uniformiser of $L$.
\end{proof}

\theoremstyle{remark}\newtheorem{suffice}[general element]{Remark}
\begin{suffice}
\label{suffice}
In going from congruence \eqref{eq:1st congruence} to congruence \eqref{eq:general element} we have exchanged an equivalence modulo $\pi_{\Kk}$ for an equivalence modulo $\pi_{\Ll}$. But the congruence modulo $\pi_{\Ll}$ is still strong enough to completely determine the local reciprocity map $\theta_{\LK}$. This rigidity follows from the fact that the extension $L/K$ is \emph{tamely} ramified, i.e. the ramification group
$$G_1=\left\{g\in\mathrm{Gal}(L/K)\Big| g(x)\equiv x \ \mathrm{mod}\ \pi_{\Ll}^2\ \ \forall x\in \cO_{\Ll}\right\}\ \ \textrm{is trivial.}$$
Recall that, since $L/K_{nr}$ is totally ramified, $\cO_{\Ll}=\cO_{\Kk_{nr}}[\pi_{\Ll}]$ and therefore
$$G_1=G_0\cap \left\{g\in\mathrm{Gal}(L/K)\Big| g(\pi_{\Ll}){\pi_{\Ll}}^{-1}\equiv 1\ \mathrm{mod}\ \pi_{\Ll}\right\}$$
where the inertia group
\begin{eqnarray*}
G_0 & = & \left\{g\in\mathrm{Gal}(L/K)\Big| g(x)\equiv x \ \mathrm{mod}\ \pi_{\Ll}\ \ \forall x\in \cO_{\Ll}\right\}=\mathrm{Gal}(L/K_{nr})\\
\ & = & \left\{g\in \mathrm{Gal}(L/K)\Big| g(y)\equiv y\ \mathrm{mod}\ \pi_{\Ll}\ \ \forall y\in \cO_{\Ll}^*\right\}
\end{eqnarray*}
The last equality holds because any element $g\in\mathrm{Gal}(L/K)$ preserves valuations so the condition $g(x)\equiv x \ \mathrm{mod}\ \pi_{\Ll}\ \ \forall x\in \pi_{\Ll}\cO_{\Ll}$ is trivially satisfied.

This tells us that
$$\left\{g\in\mathrm{Gal}(L/K)\Big| g(x){x}^{-1}\equiv 1\ \mathrm{mod}\ \pi_{\Ll}\ \ \forall\ 0\neq x\in\cO_{\Ll}\right\}=G_1=\{1\}.$$
In other words, the congruences \eqref{eq:general element} for all $0\neq\beta\in \cO_{\Ll}$ are enough to specify the element $\sigma_{\pi_{\Kk}}\in\mathrm{Gal}(L/K)$.
\end{suffice}
\theoremstyle{plain}\newtheorem{general congruence}[general element]{Proposition}

\begin{general congruence}
\label{general congruence}
For all $i\in\bbN$, for all $u\in\cO_{\Kk}^*$ and for all $\beta\in L^*$
\begin{equation}
\label{eq:general congruence}
\frac{\sigma_{u\pi_{\Kk}^i}(\beta)}{\beta}\equiv \frac{\beta^{(p^{ti}-1)}}{((-1)^{(e-1)}\pi_{\Kk})^{(p^{ti}-1)\mathfrak{v}_{\Kk}(\beta)}u^{(p^t-1)\mathfrak{v}_{\Kk}(\beta)}}\ \ \mathrm{mod}\ \pi_{\Ll}.
\end{equation}
\end{general congruence}
We will return to the proof of Proposition \ref{general congruence} after proving the following lemmas.

\newtheorem{pipower congruence}[general element]{Lemma}
\begin{pipower congruence}
\label{pipower congruence}
For all $i\in\bbN$ and for all $\beta\in L^*$
$$\frac{\sigma_{\pi_{\Kk}^i}(\beta)}{\beta}\equiv \frac{\beta^{(p^{ti}-1)}}{((-1)^{(e-1)}\pi_{\Kk})^{(p^{ti}-1)\mathfrak{v}_{\Kk}(\beta)}}\ \ \mathrm{mod}\ \pi_{\Ll}.$$
\end{pipower congruence}

\begin{proof}
We prove this by induction on $i$. Lemma \ref{general element} gives the result for $i=1$. Suppose the claim holds for some $j\in\bbN$. Then
$$\frac{\sigma_{\pi_{\Kk}^{j+1}}(\beta)}{\beta}=\frac{\sigma_{\pi_{\Kk}}\circ\sigma_{\pi_{\Kk}^{j}}(\beta)}{\beta}
=\sigma_{\pi_{\Kk}}\left(\frac{\sigma_{\pi_{\Kk}^j}(\beta)}{\beta}\right)\frac{\sigma_{\pi_{\Kk}}(\beta)}{\beta}.$$
By the induction hypothesis and Lemma \ref{general element} we have
\begin{eqnarray*}
& & \sigma_{\pi_{\Kk}}\left(\frac{\sigma_{\pi_{\Kk}^j}(\beta)}{\beta}\right)\frac{\sigma_{\pi_{\Kk}}(\beta)}{\beta}\\ & \equiv &
\sigma_{\pi_{\Kk}}\left( \frac{\beta^{(p^{tj}-1)}}{((-1)^{(e-1)}\pi_{\Kk})^{(p^{tj}-1)\mathfrak{v}_{\Kk}(\beta)}}\right)\frac{\sigma_{\pi_{\Kk}}(\beta)}{\beta}\\
\\
 &  \equiv &  \frac{{\sigma_{\pi_{\Kk}}(\beta)}^{p^{tj}-1}}{((-1)^{(e-1)}\pi_{\Kk})^{(p^{tj}-1)\mathfrak{v}_{\Kk}(\beta)}}\ \frac{\sigma_{\pi_{\Kk}}(\beta)}{\beta}\\ \\
   & \equiv & \left(\frac{\sigma_{\pi_{\Kk}}(\beta)}{\beta}\right)^{p^{tj}}\frac{\beta^{(p^{tj}-1)}}{((-1)^{(e-1)}\pi_{\Kk})^{(p^{tj}-1)\mathfrak{v}_{\Kk}(\beta)}}\ \\ \\
    & \equiv & \left(\frac{\beta^{(p^t-1)}}{((-1)^{(e-1)}\pi_{\Kk})^{(p^t-1)\mathfrak{v}_{\Kk}(\beta)}}\right)^{p^{tj}}\frac{\beta^{(p^{tj}-1)}}{((-1)^{(e-1)}\pi_{\Kk})^{(p^{tj}-1)\mathfrak{v}_{\Kk}(\beta)}}\ \\ \\
     & \equiv & \frac{\beta^{(p^{(j+1)t}-1)}}{((-1)^{(e-1)}\pi_{\Kk})^{(p^{(j+1)t}-1)\mathfrak{v}_{\Kk}(\beta)}}
  \ \ \mathrm{mod}\ \pi_{\Ll}
\end{eqnarray*}
\end{proof}

\theoremstyle{remark}\newtheorem{any unif}[general element]{Remarks}
\begin{any unif}
\label{any unif}
\begin{enumerate}
\item{}Notice that the uniformiser $\pi_{\Kk}$ was chosen arbitrarily, hence the same congruences apply for any given uniformiser of $K$. I.e. if $\pi_{\Kk}'$ is another uniformiser of $K$, then for all $i\in\bbN$ and for all $\beta\in L^*$
    $$\frac{\sigma_{\pi_{\Kk}'^i}(\beta)}{\beta}\equiv \frac{\beta^{(p^{ti}-1)}}{\left((-1)^{(e-1)}\pi_{\Kk}'\right)^{(p^{ti}-1)\mathfrak{v}_{\Kk}(\beta)}}\ \ \mathrm{mod}\ \pi_{\Ll}.$$

\item{}The group $K^*/N_{\LK}(L^*)$ is cyclic of order $q^n$ so
$\sigma_{\pi_{\Kk}^{-1}}=\sigma_{\pi_{\Kk}^{q^n-1}}.$
\end{enumerate}
\end{any unif}

\theoremstyle{plain}\newtheorem{unit congruence}[general element]{Lemma}

\begin{unit congruence}
\label{unit congruence}
For all $u\in\cO_{\Kk}^*$ and for all $\beta\in L^*$

$$\frac{\sigma_{u}(\beta)}{\beta}\equiv \frac{1}{u^{(p^t-1)\mathfrak{v}_{\Kk}(\beta)}}\ \ \mathrm{mod}\ \pi_{\Ll}.$$
\end{unit congruence}

\begin{proof}
Write $u=\frac{\pi_{\Kk}'}{\pi_{\Kk}}$ where $\pi_{\Kk}'$ and $\pi_{\Kk}$ are uniformisers of $K$. Now
$$\frac{\sigma_{u}(\beta)}{\beta}=\frac{\sigma_{\pi_{\Kk}^{-1}}\circ\sigma_{\pi_{\Kk}'}(\beta)}{\beta}=\frac{\sigma_{\pi_{\Kk}^{q^n-1}}\circ\sigma_{\pi_{\Kk}'}(\beta)}{\beta}=\sigma_{\pi_{\Kk}^{q^n-1}}\left(\frac{\sigma_{\pi_{\Kk}'}(\beta)}{\beta}\right)\frac{\sigma_{\pi_{\Kk}^{q^n-1}}(\beta)}{\beta}$$

Applying Lemma \ref{pipower congruence}, we see that
\begin{eqnarray*}
& & \sigma_{\pi_{\Kk}^{q^n-1}}\left(\frac{\sigma_{\pi_{\Kk}'}(\beta)}{\beta}\right)\frac{\sigma_{\pi_{\Kk}^{q^n-1}}(\beta)}{\beta}\\
& \equiv & \sigma_{\pi_{\Kk}^{q^n-1}}\left(\frac{\beta^{(p^t-1)}}{\left((-1)^{(e-1)}\pi_{\Kk}'\right)^{(p^t-1)\mathfrak{v}_{\Kk}(\beta)}}\right)\frac{\sigma_{\pi_{\Kk}^{q^n-1}}(\beta)}{\beta}\\ \\
 & \equiv & \left(\frac{\sigma_{\pi_{\Kk}^{q^n-1}}(\beta)}{\beta}\right)^{p^t} \frac{\beta^{(p^{t}-1)}}{\left((-1)^{(e-1)}\pi_{\Kk}'\right)^{(p^t-1)\mathfrak{v}_{\Kk}(\beta)}}\\ \\
 & \equiv & \left(\frac{\beta^{(p^{(q^n-1)t}-1)}}{((-1)^{(e-1)}\pi_{\Kk})^{(p^{(q^n-1)t}-1)\mathfrak{v}_{\Kk}(\beta)}}\right)^{p^t} \frac{\beta^{(p^{t}-1)}}{\left((-1)^{(e-1)}\pi_{\Kk}'\right)^{(p^{t}-1)\mathfrak{v}_{\Kk}(\beta)}}\\ \\
  & \equiv & \frac{\beta^{(p^{q^nt}-1)}}{((-1)^{(e-1)}\pi_{\Kk})^{(p^{q^nt}-1)\mathfrak{v}_{\Kk}(\beta)}}\ \frac{1}{u^{(p^t-1)\mathfrak{v}_{\Kk}(\beta)}}\ \ \mathrm{mod}\ \pi_{\Ll}.
\end{eqnarray*}
So we have reduced the problem to showing that
$$\frac{\beta^{(p^{q^nt}-1)}}{((-1)^{(e-1)}\pi_{\Kk})^{(p^{q^nt}-1)\mathfrak{v}_{\Kk}(\beta)}}\equiv 1 \ \ \mathrm{mod}\ \pi_{\Ll}.$$
Now recall that the residue field of $L$, denoted $\ell$, has $p^{q^{(n-d)}t}$ elements. Thus $\ell^*$ is a multiplicative group of order $p^{q^{(n-d)}t}-1$. Write
$$M=\frac{p^{q^nt}-1}{p^{q^{(n-d)}t}-1}=1+p^{q^{(n-d)}t}+\left(p^{q^{(n-d)}t}\right)^2+\ldots +\left(p^{q^{(n-d)}t}\right)^{q^d-1}$$
and note that, since $q^d$ divides $p^t-1$, the right-hand side is divisible by $q^d$. Hence for any $\beta\in L^*$, $M\mathfrak{v}_{\Kk}(\beta)\in\bbZ$ and it makes sense to write
\begin{eqnarray*}
 \frac{\beta^{(p^{q^nt}-1)}}{((-1)^{(e-1)}\pi_{\Kk})^{(p^{q^nt}-1)\mathfrak{v}_{\Kk}(\beta)}} & = & \left(\frac{\beta^M}{((-1)^{(e-1)}\pi_{\Kk})^{M\mathfrak{v}_{\Kk}(\beta)}}\right)^{p^{q^{(n-d)}t}-1}\\
& \equiv & 1 \ \ \mathrm{mod}\ \pi_{\Ll}.
\end{eqnarray*}
\end{proof}
We now return to the proof of Proposition \ref{general congruence}.

\begin{proof}
Using Lemmas \ref{pipower congruence} and \ref{unit congruence}, we see that
\begin{eqnarray*}
\frac{\sigma_{u\pi_{\Kk}^i}(\beta)}{\beta} & =  & \sigma_{u}\left(\frac{\sigma_{\pi_{\Kk}^i}(\beta)}{\beta}\right)\frac{\sigma_u(\beta)}{\beta}\\ \\
 & \equiv & \sigma_u\left(\frac{\beta^{(p^{ti}-1)}}{((-1)^{(e-1)}\pi_{\Kk})^{(p^{ti}-1)\mathfrak{v}_{\Kk}(\beta)}}\right)\frac{\sigma_u(\beta)}{\beta}\\ \\
  & \equiv & \left(\frac{\sigma_u(\beta)}{\beta}\right)^{p^{ti}}\frac{\beta^{(p^{ti}-1)}}{((-1)^{(e-1)}\pi_{\Kk})^{(p^{ti}-1)\mathfrak{v}_{\Kk}(\beta)}}\\ \\
   & \equiv & \left(\frac{1}{u^{(p^t-1)\mathfrak{v}_{\Kk}(\beta)}}\right)^{p^{ti}}\frac{\beta^{(p^{ti}-1)}}{((-1)^{(e-1)}\pi_{\Kk})^{(p^{ti}-1)\mathfrak{v}_{\Kk}(\beta)}}\ \ \mathrm{mod}\ \pi_{\Ll}.
\end{eqnarray*}
Note that the reduction of $u\in\cO_{\Kk}^*$ modulo $\pi_{\Ll}$ is an element of $k\subset\ell$, where $k$ denotes the residue field of $K$. Recall that $|k|=p^t$, hence $u^{p^{ti}}\equiv u \ \mathrm{mod}\ \pi_{\Kk}$ and therefore
$$  \left(\frac{1}{u^{(p^t-1)\mathfrak{v}_{\Kk}(\beta)}}\right)^{p^{ti}}
\equiv \frac{1}{u^{(p^t-1)\mathfrak{v}_{\Kk}(\beta)}} \ \ \mathrm{mod}\ \pi_{\Ll}$$
as required.
\end{proof}
\section{The cyclic totally ramified case}
\label{totramcase}
Once again, let $K$ be a local field and let $L/K$ be a tamely ramified cyclic extension of prime-power degree. Now we consider the case where the extension $L/K$ is totally tamely ramified. In this case, Proposition \ref{ethroot} tells us that $K$ contains a primitive root of unity of degree $[L:K]$. This means that local reciprocity can be explicitly calculated via the Hilbert norm residue symbol, for which explicit formulae have been given by Iwasawa \cite{Iwasawa}, Serre \cite{Serre}, Fesenko-Vostokov \cite{Fesenko-Vostokov}, De Shalit \cite{De Shalit} and others. Nevertheless, we include this case for completeness. Our aim is to determine local reciprocity and show that the formula of Proposition \ref{general congruence} still applies. The new situation is as follows:

\subsection{Situation}
\label{ramsituation}
\begin{itemize}
\item[-] the residue field of $K$, denoted $k$, has $|k|=p^t$
\item[-]$L/K$ cyclic totally ramified extension of local fields with\\ \mbox{$[L:K]=q^{n}$} where $q$ and $p$ are distinct primes
\end{itemize}
$L/K$ is totally tamely ramified of degree $q^n$ so we know from Proposition \ref{ethroot} that $K$ contains the $q^n$th roots of unity. Thus Kummer theory allows us to write $L=K(\sqrt[q^n]{\pi_{\Kk}})$ for some uniformiser $\pi_{\Kk}$ of $K$. To simplify notation we write $\delta=\sqrt[q^n]{\pi_{\Kk}}$.

\subsection{A cyclic algebra}
Choose a generator $\sigma$ of $\mathrm{Gal}(L/K)$. $\sigma$ acts on $\delta$ as multiplication by some primitive $q^n$th root of unity which we denote by $\zeta_{q^n}$. Let $\chi\in \mathrm{Hom}(G,\bbQ/{\bbZ})$ be the unique character such that $\chi : \sigma\mapsto\frac{1}{q^n}\ \mathrm{mod}\ \bbZ$.\\Choose an element $\eta\in\cO_{\Kk}^*$ whose reduction modulo $\pi_{\Kk}$ (which we denote by $\overline{\eta}$) generates the cyclic group $k^*$ and construct the cyclic algebra $B=(\chi,\eta)$ as follows:
$$ B=(\chi, \eta)=\left\{\sum_{i=0}^{q^n-1}a_i v^i\ \Big|\ a_i\in L=K(\delta)\right\}$$
with multiplication given by
\begin{itemize}
\item []$v^{q^n}=\eta$
\item []$v^i a=\sigma^i(a)v^i\ \forall\ a\in L$.
\end{itemize}
$B$ is a central simple algebra over $K$ which is split by $L/K$. Hereafter, our general strategy will be identical to that of Section \ref{mixedcase}. We will calculate the Hasse invariant of $[B]$ and use this to determine the local reciprocity map $\theta_{\LK}$. The only difference will be the way in which we construct the unramified splitting field for $B$.

\subsection{An unramified splitting field for $B$}

\newtheorem{ramunram}{Lemma}[subsection]
\begin{ramunram}
\label{ramunram}
Define $F=K(v)$. Then $F/K$ is an unramified Galois extension.
\end{ramunram}

\begin{proof}
$v$ is a root of the polynomial $g(x)=x^{q^n}-\eta\in K[x]$. Since $\eta\in\cO_{\Kk}^*$ and $q^n$ is coprime to the residue characteristic, Hensel's Lemma ensures that $F/K$ is an unramified extension. But then $F/K$ must be Galois since it is an unramified extension and we have $\mathrm{Gal}(F/K)\cong\mathrm{Gal}(\mathfrak{f}/k)$.
\end{proof}
Using the multiplication rules for $B$ described above, we see that
$$\delta^i v\delta^{-i}=\delta^i \sigma(\delta)^{-i}v=\zeta_{q^n}^{-i}v$$
for all $i\in\bbN$.

Let $\tau\in\mathrm{Gal}(F/K)$ denote the generator which maps any $x\in F$ to $\delta x {\delta}^{-1}$. Let $\phi\in \mathrm{Hom}(\mathrm{Gal}(F/K),\bbQ/{\bbZ})$ be the unique character such that $\phi:\tau\mapsto\frac{1}{q^n}\ \mathrm{mod}\ \bbZ$. Recall the definition of the cyclic algebra $B=(\chi,\eta)$ and rewrite $B$ in the following way:
\begin{eqnarray*}
B=(\chi, \eta) & = & \left\{\sum_{i=0}^{q^n-1}a_i v^i\ \Big|\ a_i\in L=K(\delta)\right\}\\
 & = & \left\{\sum_{i=0}^{q^n-1}c_i \delta^i\ \Big|\ c_i\in F=K(v)\right\}=(\phi,\pi_{\Kk})
\end{eqnarray*}
where
\begin{itemize}
\item []$\delta^{q^n}=\pi_{\Kk}$
\item []$\delta^k c=\tau^k(c)\delta^k\ \forall \ c\in F$.
\end{itemize}
Note that considering the dimension of $B$ as a vector space over $K$ gives $$q^{2n}=q^n[L:K]=[B:K]=[B:F][F:K]=q^n[F:K]$$
which justifies the implicit assumption that $[F:K]=q^n$.

We can now use this to compute the Hasse invariant of $B$:
$$\mathrm{inv}_{\Kk}(\chi,\eta)=\mathrm{inv}_{\Kk}(\phi,\pi_{\Kk})=\mathfrak{v}_{\Kk}(\pi_{\Kk})\phi\left(\mathrm{Frob}_{\FK}\right)$$
where $\mathfrak{v}_{\Kk}$ denotes the valuation on $K$, normalised so that $\mathfrak{v}_{\Kk}(\pi_{\Kk})=1$. Observe that $\phi(\mathrm{Frob}_{\FK})=\frac{r}{q^n}$ where $r\in \bbN$ is such that $\tau^r=\mathrm{Frob}_{\FK}$.

\subsection{Determining local reciprocity}
We are looking for an $r\in \bbN$ such that $\tau^r$ has the same action as $\mathrm{Frob}_{\FK}$ on the residue field of $F$. In other words, we need $r\in\bbN$ such that:
$$\tau^r(v)v^{-1}\equiv v^{p^t-1}\equiv \eta^{\frac{p^t-1}{q^n}}\ \ \mathrm{mod}\ \pi_{\Kk}$$
where $p^t=|k|$ is the number of elements in the residue field of $K$. We saw in Proposition \ref{ethroot} that $e=e_{\LK}=q^n$ divides $p^t-1$ since the group of $e$th roots of unity injects into $k^*$.

Recall from the definition of $\tau$ that $$\tau^r(v)v^{-1}=\zeta_{q^n}^{-r}=\left(\sigma^r(\delta)\delta^{-1}\right)^{-1}.$$
Thus the local reciprocity map $\theta_{L/K}$ is determined by the following action on $\eta$:
$$\theta_{L/K}:\eta\mapsto \sigma^r$$
where $r\in \bbN$ is such that $$\zeta_{q^n}^{r}=\sigma^r(\delta)\delta^{-1}\equiv \eta^{-\left(\frac{p^t-1}{q^n}\right)}\ \ \mathrm{mod}\ \pi_{\Kk}.$$

Note that $\tau^r=\mathrm{Frob}_{\FK}$ is another generator of $\mathrm{Gal}(F/K)$. Hence $r$ must be coprime to $[F:K]=q^n=[L:K]$. Thus $\sigma^r$ is a generator of $\mathrm{Gal}(L/K)$.

\subsection{Congruence conditions}In the notation of Section \ref{congruence conditions}, we have
$$\frac{\sigma_{\eta}(\delta)}{\delta}\equiv \frac{1}{\eta^{(p^t-1)\mathfrak{v}_{\Kk}(\delta)}}\ \ \mathrm{mod}\ \pi_{\Kk}.$$
Observe that this is precisely the congruence predicted by Proposition \ref{general congruence} in the mixed case. In fact, we will show that the congruences (\ref{eq:general congruence}) of Proposition \ref{general congruence} are also true in the totally ramified case.

First note that, since $L/K$ is totally ramified, $\cO_{\Ll}=\cO_{\Kk}[\delta]$ and therefore for any $g\in\mathrm{Gal}(L/K)$
$$g(w)\equiv w \ \mathrm{mod}\ \pi_{\Ll}\ \forall\ w\in\cO_{\Ll}^*.$$

The following Lemma is now immediate:
\newtheorem{ramgeneral element}{Lemma}[subsection]
\begin{ramgeneral element}
\label{ramgeneral element}
For all $\beta\in L^*$
\begin{equation*}
\frac{\sigma_{\eta}(\beta)}{\beta}\equiv \frac{1}{\eta^{(p^t-1)\mathfrak{v}_{\Kk}(\beta)}}\ \ \mathrm{mod}\ \pi_{\Ll}.
\end{equation*}
\end{ramgeneral element}
\begin{flushright}
$\Box$
\end{flushright}

\newtheorem{ramnorms}[ramgeneral element]{Lemma}
\begin{ramnorms}
\label{ramnorms}
A unit $u\in\cO_{\Kk}^*$ is the norm of some element of $L^*$ if and only if $u^{\frac{p^t-1}{q^n}}\equiv1\ \mathrm{mod}\ \pi_{\Kk}$.
\end{ramnorms}
\begin{proof}
$L/K$ is tamely ramified of prime-power degree, so \mbox{Proposition \ref{U1norms}} tells us that $1+\pi_{\Kk}\cO_{\Kk}\subset N_{\LK}(L^*)$.

The subset of units which are norms is the kernel of the natural projection $\cO_{\Kk}^*\twoheadrightarrow \frac{\cO_{\Kk}^*}{N_{\LK}(L^*)\cap\cO_{\Kk}^*}$. This map factors as
$$\cO_{\Kk}^*\twoheadrightarrow\frac{\cO_{\Kk}^*}{1+\pi_{\Kk}\cO_{\Kk}}\twoheadrightarrow\frac{\cO_{\Kk}^*}{N_{\LK}(L^*)\cap\cO_{\Kk}^*}.$$
Recall from Proposition \ref{Ui's} that $\frac{\cO_{\Kk}^*}{1+\pi_{\Kk}\cO_{\Kk}}$ is isomorphic to $k^*$ and is therefore a cyclic group of order $p^t-1$. In addition, we saw in \mbox{Proposition \ref{ethroot}} that when $L/K$ is totally tamely ramified the Snake Lemma gives an isomorphism $$\frac{\cO_{\Kk}^*}{N_{\LK}(L^*)\cap\cO_{\Kk}^*}\ \tilde{\rightarrow}\ \frac{K^*}{N_{\LK}(L^*)}.$$ The existence of the local reciprocity map tells us that in our case the latter is a cyclic group of order $q^n$. Therefore the kernel of the natural projection $\frac{\cO_{\Kk}^*}{1+\pi_{\Kk}\cO_{\Kk}}\twoheadrightarrow\frac{\cO_{\Kk}^*}{N_{\LK}(L^*)\cap\cO_{\Kk}^*}$ is the unique subgroup of order $\frac{p^t-1}{q^n}$ of $\frac{\cO_{\Kk}^*}{1+\pi_{\Kk}\cO_{\Kk}}$. Thus our claim is proved.
\end{proof}

Observe that the local reciprocity map sends $\eta$ to a generator of $\mathrm{Gal}(L/K)$. Thus $\eta$ must generate $K^*/N_{\LK}(L^*)$. For any $u\in\cO_{\Kk}^*$, there exists $m\in\bbN$ such that $u\eta^{-m}\in N_{\LK}(L^*)$ and hence $u^{\frac{p^t-1}{q^n}}\equiv \eta^{\frac{m(p^t-1)}{q^n}}\ \mathrm{mod}\ \pi_{\Kk}$.

Then, by Lemma \ref{ramnorms}, $\sigma_u=\sigma_{\eta^m}=\sigma_{\eta}^m$.

\newtheorem{etapower}[ramgeneral element]{Lemma}
\begin{etapower}
\label{etapower}
For all $m\in\bbN$ and for all $\beta \in L^*$
$$\frac{\sigma_{\eta^m}(\beta)}{\beta}=\frac{\sigma_{\eta}^m(\beta)}{\beta}\equiv \frac{1}{\eta^{m(p^t-1)\mathfrak{v}_{\Kk}(\beta)}}\ \ \mathrm{mod}\ \pi_{\Ll}.$$

\end{etapower}
\begin{proof}
This is clear by induction on $m$.
\end{proof}

\newtheorem{ramunit}[ramgeneral element]{Corollary}
\begin{ramunit}
\label{ramunit}
For all $u\in\cO_{\Kk}^*$ and for all $\beta\in L^*$
$$\frac{\sigma_u(\beta)}{\beta}\equiv \frac{1}{u^{(p^t-1)\mathfrak{v}_{\Kk}(\beta)}}\ \ \mathrm{mod}\ \pi_{\Ll}.$$
\end{ramunit}
\begin{proof}
This follows immediately from Lemma \ref{etapower} and the preceding discussion.
\end{proof}
Since $\cO_{\Ll}=\cO_{\Kk}[\delta]$, for any $w\in\cO_{\Ll}^*$ we have
\begin{equation}
\label{unitsdie}
w^{p^t-1}\equiv 1\ \mathrm{mod}\ \pi_{\Ll}.
\end{equation}
Also, $\mathfrak{v}_{\Kk}(\delta)=\frac{1}{q^{n}}$ and $\delta^{q^n}=\pi_{\Kk}$ so $$\frac{\delta^{(p^t-1)}}{\pi_{\Kk}^{(p^t-1)\mathfrak{v}_{\Kk}(\delta)}}=1.$$
Writing any $\beta\in L^*$ as
$$\beta = w\delta^{\mathfrak{v}_{\Ll}(\beta)}=w\delta^{q^n\mathfrak{v}_{\Kk}(\beta)}$$
for $w\in\cO_{\Ll}^*$ and applying congruence (\ref{unitsdie}), we see that in fact
$$\frac{\beta^{(p^t-1)}}{\pi_{\Kk}^{(p^t-1)\mathfrak{v}_{\Kk}(\beta)}}=1$$
for all $\beta\in L^*$.

Observe that $N_{\LK}(\delta)=(-1)^{(e-1)}\pi_{\Kk}$ and hence $\sigma_{(-1)^{(e-1)}\pi_{\Kk}}$ is trivial.

Together, the facts listed above give the following result:
\newtheorem{rampipower}[ramgeneral element]{Lemma}
\begin{rampipower}
\label{rampipower}
For all $i\in\bbN$ and for all $\beta\in L^*$
$$\frac{\sigma_{\left(N_{\LK}(\delta)\right)^i}(\beta)}{\beta}\equiv 1 \equiv \frac{\beta^{(p^{ti}-1)}}{((-1)^{(e-1)}N_{\LK}(\delta))^{(p^{ti}-1)\mathfrak{v}_{\Kk}(\beta)}}\ \ \mathrm{mod}\ \pi_{\Ll}$$
where $N_{\LK}(\delta)=(-1)^{(e-1)}\pi_{\Kk}$ is our chosen uniformiser of $K$.
\end{rampipower}
\begin{flushright}
$\Box$
\end{flushright}
Now if $\pi_{\Kk}'$ is any uniformiser of $K$, we can write $\pi_{\Kk}'=uN_{\LK}(\delta)$ for some $u\in\cO_{\Kk}^*$. Thus for all $i\in\bbN$ and for all $\beta\in L^*$ we have
\begin{eqnarray*}
\frac{\sigma_{\pi_{\Kk}'^i}(\beta)}{\beta} & = & \frac{\sigma_{u^i\left(N_{\LK}(\delta)\right)^i}(\beta)}{\beta}\\ & =
& \sigma_{u^i}\left(\frac{\sigma_{\left(N_{\LK}(\delta)\right)^i}(\beta)}{\beta}\right)\frac{\sigma_{u^i}(\beta)}{\beta}\\
 & \equiv & \frac{\sigma_{u^i}(\beta)}{\beta} \\
& \equiv & \frac{1}{u^{i(p^t-1)\mathfrak{v}_{\Kk}(\beta)}}\\
 & \equiv & \frac{1}{u^{(p^{ti}-1)\mathfrak{v}_{\Kk}(\beta)}}\ \ \mathrm{mod}\ \pi_{\Ll}
 \end{eqnarray*}
where the final congruence holds because $u\in\cO_{\Kk}^*$ so $u^{p^t}\equiv u\ \mathrm{mod}\ \pi_{\Kk}$ and therefore $u^{\frac{p^{ti}-1}{p^t-1}}\equiv u^i\ \mathrm{mod}\ \pi_{\Kk}$.

Now we write $\beta=\delta^{\mathfrak{v}_{\Ll}(\beta)}w=\delta^{e\mathfrak{v}_{\Kk}(\beta)}w$ for some $w\in\cO_{\Ll}^*$ and apply congruence (\ref{unitsdie}) to see that
\begin{eqnarray*}
\frac{\beta^{(p^{ti}-1)}}{\left((-1)^{(e-1)}\pi_{\Kk}'\right)^{(p^{ti}-1)\mathfrak{v}_{\Kk}(\beta)}} & = & \frac{\delta^{q^n (p^{ti}-1)\mathfrak{v}_{\Kk}(\beta)}w^{(p^{ti}-1)}}{\left((-1)^{(e-1)}\pi_{\Kk}'\right)^{(p^{ti}-1)\mathfrak{v}_{\Kk}(\beta)}}\\
 & \equiv & \left(\frac{N_{\LK}(\delta)}{\pi_{\Kk}'}\right)^{(p^{ti}-1)\mathfrak{v}_{\Kk}(\beta)}\\
 & \equiv & \frac{1}{u^{(p^{ti}-1)\mathfrak{v}_{\Kk}(\beta)}}\ \  \mathrm{mod}\ \pi_{\Ll}
\end{eqnarray*}
Combining this with Corollary \ref{ramunit} and Lemma \ref{rampipower} we see the following:

\newtheorem{ramgeneral}[ramgeneral element]{Corollary}
\begin{ramgeneral}
\label{ramgeneral}
Let $\pi_{\Kk}$ be \textbf{any} uniformiser of $K$. Then for all $i\in\bbN$, for all $u\in\cO_{\Kk}^*$ and for all $\beta\in L^*$

$$\frac{\sigma_{u\pi_{\Kk}^i}(\beta)}{\beta}\equiv \frac{\beta^{(p^{ti}-1)}}{((-1)^{(e-1)}\pi_{\Kk})^{(p^{ti}-1)\mathfrak{v}_{\Kk}(\beta)}u^{(p^t-1)\mathfrak{v}_{\Kk}(\beta)}}\ \ \mathrm{mod}\ \pi_{\Ll}.$$
\end{ramgeneral}
\begin{flushright}
$\Box$
\end{flushright}
So we have seen that congruences (\ref{eq:general congruence}) of Proposition \ref{general congruence} still hold for a totally tamely ramified cyclic extension $L/K$ of degree $q^n$.

In \mbox{Section \ref{unram case}} we will demonstrate that these same congruences also hold for $L/K$ unramified of degree $q^n$. In each case this determines the local reciprocity map completely, for reasons explained in Remark \ref{suffice}.

\section{The unramified case}
\label{unram case}
As usual, let $K$ be a local field 
with residue field $k$ of order $p^t$. Let $L/K$ be finite and unramified. Then all units of $K$ are norms of units of $L$.\footnote{See, for example, \cite{Serre}, Chapter V, Corollary to Proposition 3.} Thus $K^*/N_{\LK}(L^*)$ is generated by the class of any uniformiser $\pi_{\Kk}$ of $K$. The local reciprocity map, $\theta_{\LK}$, is determined by
$$\theta_{\LK}:\pi_{\Kk}\mapsto \mathrm{Frob}_{\LK}.$$
Observe that in this case the congruence
$$\frac{\sigma_{u\pi_{\Kk}^i}(\beta)}{\beta}\equiv \frac{\beta^{(p^{ti}-1)}}{\pi_{\Kk}^{(p^{ti}-1)\mathfrak{v}_{\Kk}(\beta)}u^{(p^t-1)\mathfrak{v}_{\Kk}(\beta)}}\ \ \mathrm{mod}\ \pi_{\Ll}$$
holds for all $i\in\bbN$, for all $u\in\cO_{\Kk}^*$ and for all $\beta\in L^*$.

This is because $u^{p^t-1}\equiv 1\ \mathrm{mod}\ \pi_{\Kk}$ for all $u\in\cO_{\Kk}^*$ and $\mathfrak{v}_{\Kk}(\beta)\in\bbZ$ for all $\beta\in L^*$ since $L/K$ is unramified. For any $u\in\cO_{\Kk}^*$, $\sigma_u$ is trivial because $u$ is a norm. Writing $\beta\in L^*$ as $\pi_{\Kk}^{\mathfrak{v}_{\Kk}(\beta)}w$ for some $w\in\cO_{\Ll}^*$, the result follows.

Therefore, if $L/K$ is unramified, congruence (\ref{eq:general congruence}) of Proposition \ref{general congruence} holds regardless of whether $[L:K]$ is coprime to $p$ (or indeed whether $[L:K]$ is a prime power). We will make use of this fact in Section \ref{abelian}.

Thus we have proved the following more general result:

\theoremstyle{plain}\newtheorem{generalthm2}{Theorem}[section]
\begin{generalthm2}
\label{generalthm2}
Let $K$ be a local field with residue field $k$ of order $p^t$ and let $L/K$ be a tamely ramified cyclic extension of degree $q^n$ with $q$ a prime distinct from $p$. Let $e=e_{\LK}$ be the ramification degree of $L/K$. Fix a uniformiser, $\pi_{\Kk}$, of $K$. Then the local reciprocity map $\theta_{\LK}$ satisfies
$$\frac{\theta_{\LK}\left(u\pi_{\Kk}^i\right)(\beta)}{\beta}\equiv \frac{\beta^{(p^{ti}-1)}}{((-1)^{e-1}\pi_{\Kk})^{(p^{ti}-1)\mathfrak{v}_{\Kk}(\beta)}u^{(p^t-1)\mathfrak{v}_{\Kk}(\beta)}}\ \ \mathrm{mod}\ \pi_{\Ll}$$
for all $i\in\bbN$, for all $u\in\cO_{\Kk}^*$ and for all $\beta\in L^*$.
\end{generalthm2}
\begin{flushright}
$\Box$
\end{flushright}
We saw in Remark \ref{suffice} that these congruences for all $0\neq\beta\in\cO_{\Ll}$ are enough to uniquely determine an element of $\mathrm{Gal}(L/K)$. Since the extension $L/K$ is cyclic, it suffices to apply the formula above to compute the image under $\theta_{\LK}$ of a generator of $K^*/{N_{\LK}(L^*)}$.

\section{Tame abelian extensions}
\label{abelian}
Let $L/K$ be an abelian tamely ramified extension of local fields. Write $K_{nr}$ for the maximal unramified extension of $K$ inside $L$. Let the residue field of $K$, denoted $k$, have $|k|=p^t$. In particular, we know from Corollary \ref{G1} that the characteristic of the residue field, $p$, is coprime to the ramification degree $e_{\LK}$.
We decompose the finite abelian group $G=\mathrm{Gal}(L/K)$ into a direct product of cyclic groups of prime-power order:
$$G=\mathrm{Gal}(L/K)\cong C_{q_1^{n_1}}\times C_{q_2^{n_2}}\times \ldots \times C_{q_r^{n_r}}$$ where the $q_j$ are (not necessarily distinct) primes.

Thus we can write any $\sigma\in G$ uniquely as $\sigma=\sigma_1\circ \sigma_2\circ\dots\circ \sigma_{r}$ where $\sigma_j\in C_{q_j^{n_j}}$.

Write $L_j=L^{C_{q_j^{n_j}}}$ for the subfield of $L$ which is fixed by $C_{q_j^{n_j}}$ and recall from Proposition \ref{diagram} that the following diagram commutes.
\begin{eqnarray}
\xymatrix{
L_j^*\ar[rr]^{\theta_{\LLj}}\ar[d]_{N_{\LjK}} && C_{q_j^{n_j}}\ar@{^{(}->}[d]\\
K^*\ar[rr]^{\theta_{\LK}} && \mathrm{Gal}(L/K)
}
\end{eqnarray}
Fix $b\in K^*$ and let $\sigma=\sigma_b$ be the image of $b$ under the local reciprocity map $\theta_{\LK}$.
Decompose $\sigma=\sigma_1\circ\dots\circ\sigma_r$ where $\sigma_j\in C_{q_j^{n_j}}$.

Note that we allow the $q_j$ to equal $p$ but if this is the case we must have $e_{\LLj}=1$, since $p$ does not divide $e_{\LK}=e_{\LLj}e_{\LjK}$. In other words, the extensions $L/L_j$ with Galois group $C_{q_j^{n_j}}=C_{p^{n_j}}$ are all unramified.
Therefore, as stated in Section \ref{unram case}, the congruences of Theorem \ref{generalthm2} hold for these $L/L_j$ with Galois group $C_{p^{n_j}}$. Furthermore, we can apply \mbox{Theorem \ref{generalthm2}} directly to the extensions $L/L_i$ with Galois group $C_{q_i^{n_i}}$ where $q_i\neq p$.

Our aim is to combine the congruences of Theorem \ref{generalthm2} for all the extensions $L/L_j$ and obtain congruences which hold for $L/K$.

\newtheorem{abelianthm}{Theorem}[section]
\begin{abelianthm}
\label{abelianthm}
Let $K$ be a local field with residue field $k$ of order $p^t$ and let $L/K$ be a tamely ramified finite abelian extension. Let $e=e_{\LK}$ be the ramification degree of $L/K$. Fix a uniformiser, $\pi_{\Kk}$, of $K$. Then the local reciprocity map $\theta_{\LK}$ satisfies
$$\frac{\theta_{\LK}\left(u\pi_{\Kk}^i\right)(\beta)}{\beta}\equiv \frac{\beta^{(p^{ti}-1)}}{((-1)^{e-1}\pi_{\Kk})^{(p^{ti}-1)\mathfrak{v}_{\Kk}(\beta)}u^{(p^t-1)\mathfrak{v}_{\Kk}(\beta)}}\ \ \mathrm{mod}\ \pi_{\Ll}$$
for all $i\in\bbN$, for all $u\in\cO_{\Kk}^*$ and for all $\beta\in L^*$.
\end{abelianthm}

We will need the following lemmas.
\newtheorem{unitnorm}[abelianthm]{Lemma}
\begin{unitnorm}
 \label{unitnorm}
Let $L/K$ be a finite extension of local fields and let $u\in\cO_{\Ll}^*$. Let $e=e_{\LK}$ be the ramification index of $L/K$ and let $f=f_{\LK}$ be the residue degree of $L/K$. Let $k$, $\ell$ denote the residue fields of $K$ and $L$ respectively and let $\overline{u}$ denote the image of $u$ in $\ell^*$. Then
$$N_{\LK}(u)\equiv N_{\ell/k}(\overline{u})^e\ \mathrm{mod}\ \pi_{\Ll}$$ and
$$N_{\ell/k}(\overline{u})\equiv u^{\frac{p^{tf}-1}{p^t-1}}\ \mathrm{mod}\ \pi_{\Ll}.$$
\end{unitnorm}
\begin{proof}
$N_{\LK}(u)=N_{\KnrK}\left(N_{\LKnr}(u)\right).$
Recall from Proposition \ref{G0} that
$$\mathrm{Gal}(L/K_{nr})=G_0=\{g\in G\mid \mathfrak{v}_{\scriptscriptstyle{L}}(g(z)-z)\geq 1\quad \forall z\in \mathcal{O}_L\}.$$
Thus for all $g\in \mathrm{Gal}(L/K_{nr})$ we have $g(u)\equiv u\ \mathrm{mod}\ \pi_{\Ll}$ and hence $$N_{\LKnr}(u)\equiv u^e\ \mathrm{mod}\ \pi_{\Ll}.$$
So, recalling that the residue field of $K_{nr}$ is the same as that of $L$, we see that
$$N_{\LK}(u)\equiv N_{\KnrK}(u^e)\equiv N_{\ell/k}(\overline{u})^e \ \mathrm{mod}\ \pi_{\Ll}.$$
But $\mathrm{Gal}(\ell/k)$ is cyclic of order $f$ and generated by the Frobenius automorphism which sends any $x\in \ell$ to $x^{p^t}$. This proves the second claim.
\end{proof}

\newtheorem{uniformisernorm}[abelianthm]{Lemma}
\begin{uniformisernorm}
 \label{uniformisernorm}
Let $L/K$ be a finite tamely ramified extension of local fields. Let $e=e_{\LK}$ be the ramification index of $L/K$ and let $f=f_{\LK}$ be the residue degree of $L/K$. Write $\pi_{\Ll}^e=\pi_{\Kk}u$ for some $u\in\cO_{\Ll}^*$. Let $k$, $\ell$ denote the residue fields of $K$ and $L$ respectively and let $\overline{u}$ denote the image of $u$ in $\ell^*$. Then
$$\frac{N_{\LK}(\pi_{\Ll})}{\left((-1)^{(e-1)}\pi_{\Kk}\right)^f}\equiv N_{\ell/k}(\overline{u})\equiv u^{\frac{p^{tf}-1}{p^t-1}}\ \mathrm{mod}\ \pi_{\Ll}.$$
\end{uniformisernorm}

\begin{proof}
Write $K_{nr}$ for the maximal unramified subextension of $L/K$. Reducing $\mathrm{mod}\ \pi_{\Ll}$ and recalling that the residue field of $L$ is the same as that of $K_{nr}$, we see that $u=u_1u_0$ where $u_1\equiv 1\ \mathrm{mod}\ \pi_{\Ll}$ and $u_0\in \cO_{\Kk_{nr}}^*$.

Now consider the polynomial $x^{e}-u_1$, which reduces to $x^{e}-1\ \mathrm{mod}\ \pi_{\Ll}$.
We know from Proposition \ref{ethroot} that the $e$th roots of unity are contained in $K$ and therefore in $L$. Since $e$ is coprime to the residue characteristic, we can apply Hensel's lemma and lift the roots of the polynomial $\mathrm{mod}\ \pi_\Ll$ to $e$th roots of $u_1$ in $L$.

Let $v_1\in \cO_{\Ll}^*$ be such that $v_1^e=u_1$. Thus
$$\left(\frac{\pi_{\Ll}}{v_1}\right)^e=u_0\pi_{\Kk}$$ and
$$N_{L/K_{nr}}\left(\frac{\pi_{\Ll}}{v_1}\right)=(-1)^{e-1}u_0\pi_{\Kk}.$$
Therefore
$$\frac{N_{\LK}(\pi_{\Ll})}{\left((-1)^{(e-1)}\pi_{\Kk}\right)^f}=N_{\LK}(v_1)N_{\KnrK}(u_0).$$

Lemma \ref{unitnorm} tells us that $N_{\LK}(v_1)\equiv N_{\ell/k}(\overline{v_1^e})\ \mathrm{mod}\ \pi_{\Ll}$.

But $v_1^e\equiv u_1\equiv 1 \ \mathrm{mod}\ \pi_{\Ll}$.

Finally, note that the residue field of $K_{nr}$ is the same as that of $L$ and $u_0\equiv u\ \mathrm{mod}\ \pi_{\Ll}$.
Hence
$$N_{\KnrK}(u_0)\equiv N_{\ell/k}(\overline{u})\ \mathrm{mod}\ \pi_{\Ll}$$
as required.
\end{proof}

We now return to the proof of Theorem \ref{abelianthm}. We proceed by induction on $r$, the number of components in the decomposition of the Galois group $G=\mathrm{Gal}(L/K)$ into a direct product of cyclic groups of prime-power order. In the case where the extension $L/K$ is cyclic and of prime power degree we are reduced to Theorem \ref{generalthm2}.

Suppose we know that the congruences of Theorem \ref{abelianthm} hold for any extension whose Galois group decomposes into a direct product of at most $s-1$ cyclic groups of prime-power order. Now suppose we have an extension of local fields $L/K$ such that
$$G=\mathrm{Gal}(L/K)\cong C_{q_1^{n_1}}\times C_{q_2^{n_2}}\times \ldots \times C_{q_s^{n_s}}$$ where the $q_j$ are (not necessarily distinct) primes.

Let $H=C_{q_1^{n_1}}\times C_{q_2^{n_2}}\times \ldots \times C_{q_{s-1}^{n_{s-1}}}$ so $G=H\times C_{q_s^{n_s}}$.

Let $b\in K^*$ and let $\sigma=\theta_{\LK}(b)$ be the image of $b$ under the local reciprocity map $\theta_{\LK}$.

Write $\sigma=\sigma_{\Hh}\circ\sigma_s$ where $\sigma_{\Hh}\in H$ and $\sigma_s\in C_{q_s^{n_s}}$.

Let $b_{\Hh}\in \left(L^{\Hh}\right)^*$ be such that $\theta_{\LLH}(b_{\Hh})=\sigma_{\Hh}$. Similarly, let $b_s\in L_s^*$ be such that $\theta_{\LLs}(b_s)=\sigma_s$. Recall from Proposition \ref{diagram} that $$\theta_{\LLH}(b_{\Hh})=\theta_{\LK}\left(N_{\LHK}(b_{\Hh})\right)\ \ \textrm{and}\ \ \theta_{\LLs}(b_s)=\theta_{\LK}\left(N_{\LsK}(b_s)\right).$$

Thus
$$\theta_{\LK}(b)=\sigma=\sigma_{\Hh}\circ\sigma_s=\theta_{\LK}\left(N_{\LHK}(b_{\Hh})N_{\LsK}(b_s)\right).$$

Therefore, by definition of the local reciprocity map,
$$b=N_{\LHK}(b_{\Hh})N_{\LsK}(b_s)N_{\LK}(x)$$
for some $x\in L^*$.

But $N_{\LK}(x)=N_{\LsK}\left(N_{\LLs}(x)\right)$ so we may replace $b_s$ by $b_s N_{\LLs}(x)$ and assume without loss of generality that
\begin{equation*}
\label{b decomposition}
b=N_{\LHK}(b_{\Hh})N_{\LsK}(b_s).
\end{equation*}

Pick uniformisers $\pi_{\Kk}$, $\pi_{\Hh}$ and $\pi_s$ of $K$, $L^{\Hh}$ and $L_s$ respectively and write $$b_{\Hh}=\pi_{\Hh}^iv,\ \ b_s=\pi_s^jw$$
for some $i,j\in \bbN,\ v\in\cO_{L^{\Hh}}^*,\ w\in\cO_{L_s}^*.$

Write $f_{\Hh}=f_{\LHK}$, $f_s=f_{\LsK}$ for the residue degrees of $L^{\Hh}/K$ and $L_s/K$ respectively. Similarly, write $e_{\Hh}=e_{\LHK}$, $e_s=e_{\LsK}$ for the ramification indices of $L^{\Hh}/K$ and $L_s/K$ respectively.
Thus
\begin{eqnarray*}
b & = & N_{\LHK}(\pi_{\Hh}^iv)N_{\LsK}(\pi_s^jw)\\
& = & \pi_{\Kk}^{(f_{\Hh}i+f_sj)}\left(\frac{N_{\LHK}(\pi_{\Hh})}{\pi_{\Kk}^{f_{\Hh}}}\right)^i N_{\LHK}(v)\left(\frac{N_{\LsK}(\pi_s)}{\pi_{\Kk}^{f_s}}\right)^j N_{\LsK}(w).
\end{eqnarray*}

Observe that
$$\mathfrak{v}_{\Kk}\left(N_{\LHK}(\pi_{\Hh})\right)=[L^{\Hh}:K]\mathfrak{v}_{\Kk}(\pi_{\Hh})=e_{\Hh}f_{\Hh}\frac{1}{e_{\Hh}}=f_{\Hh}=\mathfrak{v}_{\Kk}(\pi_{\Kk}^{f_{\Hh}})$$
so $\frac{N_{\LHK}(\pi_{\Hh})}{\pi_{\Kk}^{f_{\Hh}}}$ is a unit and we define
$$\cU_{\Hh}=\left(\frac{N_{\LHK}(\pi_{\Hh})}{\pi_{\Kk}^{f_{\Hh}}}\right)^i N_{\LHK}(v)\in\cO_{\Kk}^*.$$
Likewise, define
$$\cU_s=\left(\frac{N_{\LsK}(\pi_s)}{\pi_{\Kk}^{f_s}}\right)^j N_{\LsK}(w)\in\cO_{\Kk}^*$$
so that
\begin{equation}
\label{b}
b=\pi_{\Kk}^{(f_{\Hh}i+f_sj)}\cU_{\Hh}\cU_s.
\end{equation}

Write $\epsilon_{\Hh}=e_{\LLH}$, $\epsilon_{\LLs}=e_{\LLs}$ for the ramification indices of $L/L^{\Hh}$ and $L/L_s$ respectively. Let $e=e_{\LK}$ denote the ramification index of $L/K$. Thus
$$e_{\Hh}\epsilon_{\Hh}=e=e_s\epsilon_s.$$

Note that the residue field of $L^H$ has $p^{tf_{\Hh}}$ elements and the residue field of $L_s$ has $p^{tf_s}$ elements.

Now we apply the induction hypothesis to see that for all $\beta\in L^*$
\begin{eqnarray*}
\frac{\sigma_{\Hh}(\beta)}{\beta}\equiv\frac{\beta^{(p^{tf_{\Hh}i}-1)}}{((-1)^{\epsilon_{\Hh}-1}\pi_{\Hh})^{(p^{tf_{\Hh}i}-1)\mathfrak{v}_{\LH}(\beta)}v^{(p^{tf_{\Hh}}-1)\mathfrak{v}_{\LH}(\beta)}}\ \mathrm{mod}\ \pi_{\Ll}
\end{eqnarray*}
and
\begin{eqnarray*}
\frac{\sigma_s(\beta)}{\beta}
 & \equiv & \frac{\beta^{(p^{tf_sj}-1)}}{((-1)^{\epsilon_s-1}\pi_s)^{(p^{tf_sj}-1)\mathfrak{v}_{\Ls}(\beta)}w^{(p^{tf_s}-1)\mathfrak{v}_{\Ls}(\beta)}}\ \mathrm{mod}\ \pi_{\Ll}.
\end{eqnarray*}
Note that $\mathfrak{v}_{\LH}=e_{\Hh}\mathfrak{v}_{\Kk}$ and $\mathfrak{v}_{\Ls}=e_s\mathfrak{v}_{\Kk}$.

Also, $\pi_{\Hh}^{e_{\Hh}}=\pi_{\Kk}u_{\Hh}$ for some $u_{\Hh}\in \cO_{\LH}^*$ and $\pi_{s}^{e_s}=\pi_{\Kk}u_s$ for some $u_s\in \cO_{\Ls}^*$.

Hence
\begin{eqnarray*}
\frac{\sigma_s(\beta)}{\beta} & \equiv & \frac{\beta^{(p^{tf_sj}-1)}}{((-1)^{\epsilon_s-1}\pi_s)^{(p^{tf_sj}-1)\mathfrak{v}_{\Ls}(\beta)}w^{(p^{tf_s}-1)\mathfrak{v}_{\Ls}(\beta)}}\\
 & \equiv & \frac{\beta^{(p^{tf_sj}-1)}}{((-1)^{e_s(\epsilon_s-1)}\pi_{\Kk}u_s)^{(p^{tf_sj}-1)\mathfrak{v}_{\Kk}(\beta)}w^{e_s(p^{tf_s}-1)\mathfrak{v}_{\Kk}(\beta)}}\ \mathrm{mod}\ \pi_{\Ll}
\end{eqnarray*}

Using Lemma \ref{unitnorm}, we see that $N_{\LsK}(w)\equiv w^{e_s\frac{p^{tf_s}-1}{p^t-1}} \mathrm{mod}\ \pi_{\Ll}$.

Hence
$$\frac{\sigma_s(\beta)}{\beta} \equiv \frac{\beta^{(p^{tf_sj}-1)}}{((-1)^{e_s(\epsilon_s-1)}\pi_{\Kk}u_s)^{(p^{tf_sj}-1)\mathfrak{v}_{\Kk}(\beta)}N_{\LsK}(w)^{(p^t-1)\mathfrak{v}_{\Kk}(\beta)}}\ \mathrm{mod}\ \pi_{\Ll}.$$

Now write
$$\frac{p^{tf_sj}-1}{p^{tf_s}-1}=1+p^{tf_s}+p^{2tf_s}+\dots +p^{(j-1)tf_s}$$
and recall that $p^{tf_s}$ is the size of the residue field of $L_s$.

Hence \mbox{$x^{p^{tf_s}}\equiv x \ \mathrm{mod}\ \pi_{\Ll}$} for any $x\in\cO_{\Ls}^*$.

Therefore $u_s^{(p^{tf_sj}-1)\mathfrak{v}_{\Kk}(\beta)}\equiv u_s^{j(p^{tf_s}-1)\mathfrak{v}_{\Kk}(\beta)}\ \mathrm{mod}\ \pi_{\Ll}$.
Now apply \mbox{Lemma \ref{uniformisernorm}} to see that
\begin{eqnarray*}
u_s^{j(p^{tf_s}-1)\mathfrak{v}_{\Kk}(\beta)} & \equiv & \left(\frac{N_{\LsK}(\pi_s)}{\left((-1)^{(e_s-1)}\pi_{\Kk}\right)^{f_s}}\right)^{j(p^t-1)\mathfrak{v}_{\Kk}(\beta)}\\
& \equiv & (-1)^{(e_s-1)(p^{tf_sj}-1)\mathfrak{v}_{\Kk}(\beta)}\left(\frac{N_{\LsK}(\pi_s)}{\pi_{\Kk}^{f_s}}\right)^{j(p^t-1)\mathfrak{v}_{\Kk}(\beta)}\mathrm{mod}\ \pi_{\Ll}.
\end{eqnarray*}
Thus
\begin{eqnarray*}
\frac{\sigma_s(\beta)}{\beta} & \equiv &  \frac{\beta^{(p^{tf_sj}-1)}}{\left((-1)^{(e-1)}\pi_{\Kk}\right)^{(p^{tf_sj}-1)\mathfrak{v}_{\Kk}(\beta)}\cU_s^{(p^t-1)\mathfrak{v}_{\Kk}(\beta)}}\ \mathrm{mod}\ \pi_{\Ll}.
\end{eqnarray*}
Similarly,
\begin{eqnarray*}
\frac{\sigma_{\Hh}(\beta)}{\beta} & \equiv &  \frac{\beta^{(p^{tf_{\Hh}i}-1)}}{\left((-1)^{(e-1)}\pi_{\Kk}\right)^{(p^{tf_{\Hh}i}-1)\mathfrak{v}_{\Kk}(\beta)}\cU_{\Hh}^{(p^t-1)\mathfrak{v}_{\Kk}(\beta)}}\ \mathrm{mod}\ \pi_{\Ll}.
\end{eqnarray*}
Now combine the congruences above with the decomposition of $\sigma$ to see that
\begin{eqnarray*}
\frac{\sigma(\beta)}{\beta} & = & \frac{\sigma_{\Hh}\circ\sigma_s(\beta)}{\beta}\\
& = & \sigma_{\Hh}\left(\frac{\sigma_s(\beta)}{\beta}\right)\frac{\sigma_{\Hh}(\beta)}{\beta}\\
& \equiv & \frac{\beta^{(p^{tf_sj}-1)}}{\left((-1)^{(e-1)}\pi_{\Kk}\right)^{(p^{tf_sj}-1)\mathfrak{v}_{\Kk}(\beta)}\cU_s^{(p^t-1)\mathfrak{v}_{\Kk}(\beta)}}\left(\frac{\sigma_{\Hh}(\beta)}{\beta}\right)^{p^{tf_sj}}\\
& \equiv & \frac{\beta^{(p^{t(f_{\Hh}i+f_sj)}-1)}}{\left((-1)^{(e-1)}\pi_{\Kk}\right)^{\left(p^{t(f_{\Hh}i+f_sj)}-1\right)\mathfrak{v}_{\Kk}(\beta)}\left(\cU_s \cU_{\Hh}^{p^{tf_sj}}\right)^{(p^t-1)\mathfrak{v}_{\Kk}(\beta)}}\\
& \equiv & \frac{\beta^{(p^{t(f_{\Hh}i+f_sj)}-1)}}{\left((-1)^{(e-1)}\pi_{\Kk}\right)^{\left(p^{t(f_{\Hh}i+f_sj)}-1\right)\mathfrak{v}_{\Kk}(\beta)}\left(\cU_s \cU_{\Hh}\right)^{(p^t-1)\mathfrak{v}_{\Kk}(\beta)}} \ \mathrm{mod}\ \pi_{\Ll}
\end{eqnarray*}
where the last congruence holds because $\cU_{\Hh}\in\cO_{\Kk}^*$ and therefore $$\cU_{\Hh}^{p^t}\equiv \cU_{\Hh}\ \mathrm{mod}\ \pi_{\Kk}$$ since $p^t=|k|$.
Comparing with \eqref{b} concludes the proof of Theorem \ref{abelianthm}.
\begin{flushright}
$\Box$
\end{flushright}

For a fixed element $b\in K^*/N_{\LK}(L^*)$, Theorem \ref{abelianthm} gives congruences which hold for all $\beta\in L^*$. We know from Remark \ref{suffice} that these congruences for all $0\neq \beta\in\cO_{\Kk}^*$ are enough to determine $\theta_{\LK}(b)$ uniquely.

In order to compute local reciprocity for a tame finite abelian extension $L/K$, it suffices to apply Theorem \ref{abelianthm} to compute the images under $\theta_{\LK}$ of a set of elements which generate $K^*/N_{\LK}(L^*)$ as a finite group.

\end{document}